\documentclass{amsart}
\usepackage{math}
\usepackage{tikz,mathabx,mathrsfs,mathtools,amsrefs}
\usetikzlibrary{cd,quotes,angles,decorations}
\usetikzlibrary{arrows,automata}
\newenvironment{fsa}[1][auto]{\begin{tikzpicture}[->,>=stealth',
    shorten >=1pt,auto,node distance=3cm,double distance between line centers=0.45ex,
    initial text=,accepting/.style=accepting by arrow,
    every state/.style={inner sep=3pt,minimum size=0pt},
    every loop/.style={looseness=12},semithick,#1]}{\end{tikzpicture}}
\pgfdeclaredecoration{single line}{initial}{
  \state{initial}[width=\pgfdecoratedpathlength-1sp]{\pgfpathmoveto{\pgfpointorigin}}
  \state{final}{\pgfpathlineto{\pgfpointorigin}}
}

\newtheorem{mainthm}{Theorem}

\begin{document}
\title{The Thue-Morse shift, Baumslag-Solitar group, and biminimality}
\author{Laurent Bartholdi}
\address{\'Ecole Normale Sup\'erieure, Lyon\and Mathematisches Institut, Georg-August Universit\"at zu G\"ottingen}
\email{laurent.bartholdi@gmail.com}
\date{October 16th, 2019}
\begin{abstract}
  Call a group action on a topological space \emph{biminimal} if for
  any points $x,y\in X$ there exists a group element taking $x$
  arbitrarily close to $y$ and whose inverse takes $y$ arbitrarily
  close to $x$.

  A symbolic encoding of the Thue-Morse dynamical system is given, in
  terms of $\omega$-automata. It is used to prove that the Thue-Morse
  dynamical system is minimal but not biminimal.

  The $\omega$-automata also establish a link between Nekrashevych's
  presentation of limit spaces and solenoids with a construction
  described by Vershik and Solomyak.
\end{abstract}
\maketitle

\section{Introduction}
The goals of this note are threefold: develop, on one important
example, the general theory of automatic actions [in preparation]; 
give a symbolic / topological description of the Thue-Morse shift; and
answer a question of Dominik Franc\oe ur related to a strengthening of
minimality of actions.

It is my belief that salient aspects of the theory of automatic
actions become more digestible when specialized to an example, in
particular if this example is fundamental. Note that Vershik and
Solomyak gave in~\cite{vershik-solomyak:adic} an excellent account of
the Thue-Morse shift and its interpretation as an adic transformation;
however, that description is only valid measure-theoretically; by
force of some simplifications adopted in their construction, the map
they construct cannot be continuous.

I elected not to present the results in the most economical or terse
possible way, because part of my motivation is to establish
connections with the classical notation and terminology. The
fundamental messages should be: finite state automata are useful in
more than one manner, to encode in finite, easily manipulable objects
various infinite sets or languages; and it helps to keep track of the
group (or semigroup) acting on the space in question. To paraphrase
Louis Ferdinand C\'eline~\cite{celine:vabdln}*{incipit} (as suggested
by Jacques Sakarovich),
\begin{quote}
  \emph{$\ll$ Les automates, c'est l'infini mis \`a port\'ee des caniches $\gg$}.
\end{quote}

\subsection{The Thue-Morse automaton}
The Thue-Morse map, originally considered by
Thue~\cite{thue:zeichenreihen}, and Morse~\cite{morse:geodesics} in
his study of geodesics on surfaces, is the endomorphism $\zeta$ of
$\{0,1\}^*$ is given by $0\mapsto01,\;1\mapsto10$. Let
$\Omega\subset\{0,1\}^\Z$ denote those $w\in\{0,1\}^\Z$ such that all
subwords of $w$ occur in $\zeta^n(0)$ for some $n$. It is a compact,
totally disconnected set naturally carrying a $\Z$-action by
translation.

$\Omega$ may be recoded in terms of paths in a ``Bratteli-Vershik
diagram'', on which the action of $\Z$ may be defined
combinatorially. However, the classical diagram
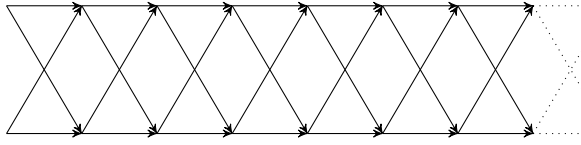
\[\begin{tikzpicture}[->,>=stealth']
    \foreach \i in {0,...,6} {
      \path (\i,0) edge +(1,0) ++(0,0) edge +(1,1.7) ++(0,1.7) edge +(1,0) ++(0,0) edge +(1,-1.7);
    }
    \clip (7,-0.1) rectangle (7.7,1.8);
    \draw[dotted] (7,0) -- +(1,0) ++(0,0) -- +(1,1.7) ++(0,1.7) -- +(1,0) ++(0,0) -- +(1,-1.7);
    
  \end{tikzpicture}
\]
is only valid measure-theoretically; there is a map
$\Omega\twoheadrightarrow\{\text{paths in the diagram above}\}$ which
is almost-everywhere bijective, but is $2:1$ on a countable set, on
which the the natural image of the shift map is discontinuous. The
otherwise-excellent reference~\cite{durand:cant} remarks dryly that
``this diagram [\dots] does not fit our setting''.

On the other hand, a powerful technique to encode dynamical systems
has been developed by Nekrashevych~\cite{nekrashevych:ssg}; in a word,
``Nekrashevych duality'' establishes an equivalence between certain
self-similar group actions and self-coverings of spaces, which doesn't
apply to $\Omega$ -- but almost does. If a picture is worth a thousand
words, the main outcome of the present paper is a finite automaton:
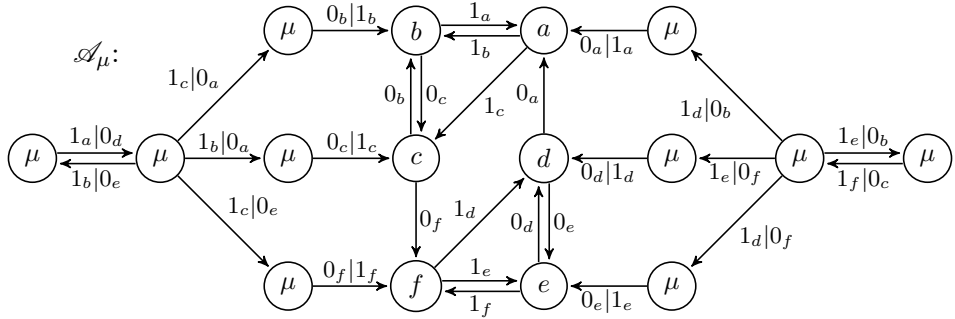
\[\begin{fsa}[every state/.style={minimum size=6mm,text width=6mm,inner sep=0mm,align=center},scale=0.85]
    \node at (-6,1.6) {\large $\mathscr A_\mu$:};
    \node[state] (a) at (1,2) {$a$};
    \node[state] (d) at (1,0) {$d$};
    \node[state] (c) at (-1,0) {$c$};
    \node[state] (f) at (-1,-2) {$f$};
    \node[state] (b) at (-1,2) {$b$};
    \node[state] (e) at (1,-2) {$e$};
    \node[state] (muca) at (-3,2) {$\mu$};
    \node[state] (muba) at (-3,0) {$\mu$};
    \node[state] (muce) at (-3,-2) {$\mu$};
    \node[state] (muad) at (-5,0) {$\mu$};
    \node[state] (mube) at (-7,0) {$\mu$};
    \node[state] (mudf) at (3,-2) {$\mu$};
    \node[state] (muef) at (3,0) {$\mu$};
    \node[state] (mudb) at (3,2) {$\mu$};
    \node[state] (mufc) at (5,0) {$\mu$};
    \node[state] (mueb) at (7,0) {$\mu$};
    \foreach\i/\j/\l in {a/c/1_c,c/f/0_f,f/d/1_d,d/a/0_a,muca/b/0_b|1_b,muba/c/0_c|1_c,muce/f/0_f|1_f,mudf/e/0_e|1_e,muef/d/0_d|1_d,mudb/a/0_a|1_a,muad/muca/1_c|0_a,muad/muba/1_b|0_a,muad/muce/1_c|0_e,mufc/mudb/1_d|0_b,mufc/muef/1_e|0_f,mufc/mudf/1_d|0_f} {
      \draw[->] (\i) -- node [inner sep=1pt] {\small $\l$} (\j);
    };
    \foreach\i/\j\l in {a/b/1_b,b/a/1_a,b/c/0_c,c/b/0_b,d/e/0_e,e/d/0_d,e/f/1_f,f/e/1_e,muad/mube/1_b|0_e,mueb/mufc/1_f|0_c,mube/muad/1_a|0_d,mufc/mueb/1_e|0_b} {
      \draw[->,decorate,decoration={single line,raise=2pt}] (\i) -- node {\small $\l$} (\j);
    };
  \end{fsa}
\]
The inner six states encode the Thue-Morse shift $\Omega$: there is an
easily-computable bijection between $\Omega$ and the set of
right-infinite paths in this $6$-vertex graph. Paths are naturally
identified with the labels read along them. The $\Z$-action is
described by the other states of $\mathscr A_\mu$: to compute the
image of a path labeled $w$, find the unique path starting at a state
labeled $\mu$ and carrying a label of the form $w|w'$ (where we
interpret labels $i_j$ as $i_j|i_j$); then the image of $w$ is $w'$.

Call $\mathscr A$ the subautomaton of $\mathscr A_\mu$ spanned by the
inner six vertices; denote by $L(\mathscr A)$ the space of infinite
paths in $\mathscr A$ (topologized by declaring paths close if they
agree on a long initial segment). We prove:
\begin{mainthm}[= Theorems~\ref{thm:factor} and~\ref{thm:homeo}]
  There is a homeomorphism $\Omega\cong L(\mathscr A)$. Using it, there
  is a natural factor map $\pi\colon\Omega\twoheadrightarrow\Z_2$,
  given by sending label $i_j$ to $i$ in $\mathscr A_\mu$, and mapping
  the shift action on $\Omega$ to the natural $\Z$-action on $2$-adic
  integers. It is generically $2:1$, and $4:1$ on $\pi^{-1}(\Z)$.

  Define $\widetilde\Z_2$ as a modification of $\Z_2$ in which the
  copy of $\Z$ is duplicated. Then there is a homeomorphism
  $\Omega\cong\Z/2\cdot\widetilde\Z_2$, a $2$-point extension of
  $\widetilde\Z_2$, under which the shift's action on $\Omega$ is
  transported to ``addition with a cocycle''
  $(s,z)\mapsto(s+\phi(z+1),z+1)$ for the map
  $\phi\colon\widetilde\Z_2\to\Z/2$ given in~\eqref{eq:phi}; and the
  factor map $\Omega\twoheadrightarrow\Z_2$ is given by the natural
  maps $\Z/2\to1$ and $\widetilde\Z_2\to\Z_2$.
\end{mainthm}

Denote by $\widehat L(\mathscr A)$ the space of bi-infinite paths in
$\mathscr A$, and write $w\sim w'$ for two bi-infinite paths in
$\mathscr A$ if there exists a bi-infinite path in $\mathscr A_\mu$
labeled $w|w'$. We construct a ``solenoid'' for the Thue-Morse shift:
a suspension of $\Omega$ on which the dynamics induced by $\zeta$
becomes invertible.
\begin{mainthm}[= Theorem~\ref{thm:solenoid}]
  The relation $\sim$ is an equivalence relation, and the quotient space
  $\mathfrak S\coloneqq\widehat L(\mathscr A)/{\sim}$
  \begin{enumerate}
  \item is compact, metrizable, connected;
  \item fibres over the circle $\R/\Z$ with fibre $L(\mathscr A)$;
    monodromy around the circle induces the shift action on the fibre;
  \item admits a quotient map to the $2$-adic solenoid
    $\mathbb S_2=(\Z_2\times\R)/(z,t+1)\sim(z+1,t)$, induced by
    forgetting alphabet decorations; fibres have cardinality $2$ or
    $4$. There is a homeomorphism
    \[\mathfrak S\longrightarrow\Z/2\cdot\bigg(\frac{\widetilde\Z_2\times[0,1]}{(z,1)\sim(z+1,0)}\bigg),\]
    on the image of which the action is given by ``addition with
    cocycle'', and on which the map
    $\mathfrak S\twoheadrightarrow\mathbb S_2$ is given by
    $\widetilde\Z_2\to\Z_2$ and $\Z/2\to1$.
  \end{enumerate}
\end{mainthm}

\subsection{Binimimal actions}
The following discussion arose during Dominik Franc\oe ur's PhD
defense~\cite{francoeur:phd}. Let $G$ be a group acting on a
topological space $\Omega$. Recall that the action is minimal if every
$G$-orbit in $\Omega$ is dense, namely if every point in $\Omega$ can
be taken into any open set by an element of $G$.

Let us call the action \emph{biminimal} if this definition can be made
symmetric: for every points with open neighbourhoods
$x\in\mathcal U\subseteq\Omega$ and $y\in\mathcal V\subseteq\Omega$,
there exists an element of $G$ taking $x$ into
$\mathcal V$ whose inverse takes $y$ into $\mathcal U$.

Define $\tau\colon X\times X\to X\times X$ by $\tau(x,y)=(y,x)$. In
case $G$ is abelian, we may let $G$ act on $X\times X$ by
$g\cdot(x,y)=(g(x),g^{-1}(y))$, and then the the action is biminimal
if $\tau$ may be approximated pointwise by the antidiagonal action of
$G$.

\begin{question}
  Is every minimal action biminimal?
\end{question}

An easy remark: if $(G,X)$ is minimal and $G$ acts on $X$ by
isometries then $(G,X)$ is biminimal. Indeed if for all $x,y\in X$ and
$r>0$ there is $g\in G$ with $g(x)\in B(y,r)$, and then
$g^{-1}(y)\in B(x,r)$.

Another easy remark, due to Volodymyr Nekrashevych : let $(G,X)$ be a
minimal action and fix $x,\mathcal U,\mathcal V$ as above. Define
$\mathcal V_0=\bigcup_{g\in G:g(x)\in\mathcal V}g(\mathcal
U)\cap\mathcal V$. Then $\mathcal V_0$ is open dense in $\mathcal V$:
given any $\mathcal W\text{ open in }\mathcal V$ there is $g\in G$
with $g(x)\in\mathcal W$ so $g(x)\in\mathcal V$ and
$\mathcal W\cap\mathcal V_0\neq\emptyset$. Therefore, if $X$ is a
Baire space then for every $x\in X$ the ``good'' $y$ in the definition
are generic (= comeagre).  However, we shall see:
\begin{mainthm}[= Theorem~\ref{thm:bimin}]\label{mainthm:bimin}
  The Thue-Morse dynamical system is minimal but not biminimal.
\end{mainthm}

Presumably, a little more work could lead to a more general statement:
\begin{conj}
  Let $\sigma$ be a minimal homeomorphism of the Cantor set. Then
  either $\sigma$ is an isometry in a metric compatible with the
  Cantor set's topology, or $\sigma$ is not biminimal.
\end{conj}
Indeed every minimal homeomorphism may be encoded by a
Bratteli-Vershik diagram, as we will see below; and Bratteli-Vershik
homeomorphisms are classified into odometers (which preserve a metric)
and expansive maps (which are thus subshifts). These correspond to
Bratteli diagrams / automata with $1$, respectively $>1$ states; and
presumably the diagrams with $>1$ states naturally come with disjoint
clopens corresponding to distinct starting states, violating the
biminimality.

\section{The Thue-Morse shift space}
Set $\Sigma_2=\{0,1\}$, and recall the Thue-Morse endomorphism $\zeta$
of $\Sigma_2^*$ given by $0\mapsto01,\;1\mapsto10$. Note that it extends
to $\Sigma_2^\N$ and $\Sigma_2^\Z$ by continuity.

Let $\Omega\subset\Sigma_2^\Z$ denote those $w\in\Sigma_2^\Z$ such
that all subwords of $w$ occur in $\zeta^n(0)$ for some $n$. It is a
compact, totally disconnected set naturally carrying a $\Z$-action. In
general, for a set $F$ and an invariant subspace $X$ of some $F^\N$ or
$F^\Z$, we denote by $\sigma$ or $\sigma_X$ the endomorphism of that
space given by $\sigma_X(u)_n=u_{n+1}$, and call the subspace a
\emph{subshift}.

Set $u=\zeta^\infty(0)=0110100110010110\dots\in\Sigma_2^\N$ a
right-infinite word, and for words $v,w\in\Sigma_2^\N$ denote by `$v.w$'
the word obtained by concatenating $v$ (in reverse) with $w$, namely
we have $(v.w)_n=w_n$ and $(v.w)_{-1-n}=v_n$ for all $n\in\N$; so
$\sigma$'s action amounts to moving the `$.$' one position to the
right. It is easy to see that we have
\[\Omega=\overline{\{\sigma^n(u.u):n\in\Z\}}.\]

It is well-known that $\Omega$ is \emph{aperiodic}, namely the
$\Z$-action is free. Moreover, there is no subword of $u$ of the form
$p q p q p$, as is shown by induction on the length of $p,q$. In fact,
a slightly larger group acts on $\Omega$: first, there is a central
involution $w\mapsto\overline w$ given by the exchange
$0\leftrightarrow 1$. Since the infinite dihedral group
$D_\infty\coloneqq\langle\beta,\gamma\mid \beta^2,\gamma^2\rangle$
acts on $\Z$ by $\beta(n)=-n$ and $\gamma(n)=-1-n$, it also acts on
$\Omega$, and the action is free except for two orbits, those of $u.u$
and $\overline{u.u}$ (both stabilized by $\gamma$); more on this later.

Recall that a dynamical system is \emph{minimal} if every orbit is
dense. It is well-known that our $\Omega$ is minimal, the criterion
being that $u.u$ is \emph{repetitive}: for every subword $v$ of $u.u$,
there is a constant $C(v)$ such that every subword of size $C(v)$ in $u.u$
contains a copy of $v$. This property in turn follows directly from
the nature of $\zeta$, with $C(v)$ growing linearly in $|v|$. Again,
more on this later.

Note furthermore that $\Omega$ admits another, non-invertible action
given by $\zeta$; really, we have an action of the ``Baumslag-Solitar
semigroup''
\[B(1,2)_+=\langle \alpha,\alpha^{-1},\zeta\mid
  \alpha^2\circ\zeta=\zeta\circ\alpha\rangle_+,
\]
with $\alpha$ acting as $\sigma$ and $\zeta$ acting as $\zeta$ on
$\Omega$. We could even include $D_\infty$ in our semigroup, and
consider
$\langle\beta,\gamma,\zeta\mid\beta^2=1,\gamma^2=1,\zeta\alpha=\alpha\zeta,\zeta\beta=\beta\alpha\beta\zeta\rangle_+$
if we wanted (we don't).

We note for future use the important property
$\Omega=\zeta(\Omega)\sqcup\sigma\zeta(\Omega)$: indeed $u$ does not
contain $000$ or $111$ as a subword, so the same holds for every
$w\in\Omega$. Thus $w$ contains a subword $01$, at either an even or
odd location. Depending on these cases, either $w$ or $\sigma(w)$ may
be factored into $01$ and $10$ subwords.

Remark that $\zeta^{-1}$ is a well-defined homeomorphism on
$\zeta(\Omega)$. We extend it to $\sigma\zeta(\Omega)$ by
$\zeta^{-1}(\sigma w)=\zeta^{-1}(w)$. We obtain a continuous, $2:1$
map $\zeta^{-1}$ on $\Omega$, satisfying $\zeta^{-1}\circ\zeta=1$ and
$\zeta\circ\zeta^{-1}(w)\in\{w,\sigma^{-1}(w)\}$ for all $w\in\Omega$.

\section{Bratteli diagrams}
Even though the $\Z$-action is very easy to understand on individual
elements of $\Omega$, it is not easy to study its dynamical
properties. For this, it is classical to re-encode $\Omega$ via
Bratteli diagrams.

\def\ozz{{}^10^0}
\def\zoz{{}^01^0}
\def\ooz{{}^11^0}
\def\zzo{{}^00^1}
\def\ozo{{}^10^1}
\def\zoo{{}^01^1}

We first extend the alphabet $\Sigma_2$ into
$\Sigma\coloneqq\{\ozz,\zoz,\ooz,\ozo,\zoo\}$, and extend the
substitution $\zeta$ to
\begin{equation}\label{eq:extendzeta}
  \begin{alignedat}{3}
  \ozz&\mapsto\zzo\;\zoz, &\qquad
  \zoz&\mapsto\ooz\;\ozz,&\qquad
  \ooz&\mapsto\zoz\;\ozz,\\
  \zzo&\mapsto\ozo\;\zoo,&
  \ozo&\mapsto\zzo\;\zoo,&
  \zoo&\mapsto\ooz\;\ozo.
\end{alignedat}
\end{equation}
What we have done is ``collared'' the original substitution; namely if
$w\in \Omega$ then we re\"encode it as $\widetilde w\in\Sigma^\Z$ by
$\widetilde w(n)={}^{w(n-1)}w(n)^{w(n+1)}$. The corresponding map
$\Omega\to\{\widetilde w:w\in \Omega\}\subset\Sigma^\Z$ is clearly a
homeomorphism on its image. Note that $\zeta$ naturally acts on
$\Sigma^\Z$, so $B(1,2)_+$ acts on $\Sigma^\Z$ preserving the image of
$\Omega$.

We finally encode $\Omega$ into $\Sigma^\N$ by
\[\lambda\colon\begin{cases}\Omega &\to\Sigma^\N,\\
    w &\mapsto((\zeta^{-n}(\widetilde w))(0))_{n\in\N}.
  \end{cases}
\]

The image of $\lambda$ is easy to understand using a graded graph
$\mathscr V^\N$ called a \emph{Bratteli diagram}, repeating
periodically in an infinite stack growing upwards the following
picture $\mathscr V$:
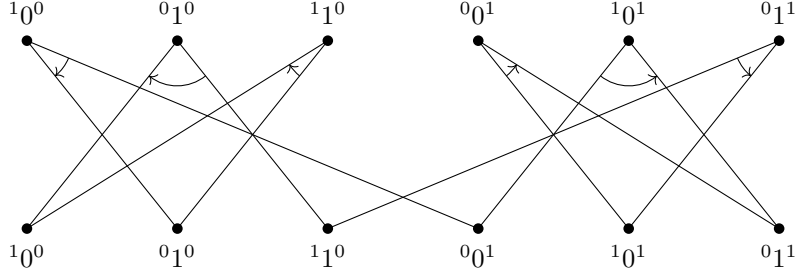
\[\begin{tikzpicture}
    \foreach\x/\y\z in {0/ozz,1/zoz,2/ooz,3/zzo,4/ozo,5/zoo} {
      \coordinate (\y-) at (2*\x,0);
      \coordinate (\y+) at (2*\x,2.5);
      \fill (\y+) circle[radius=2pt] node[above=1mm] {$\csname\y\endcsname$};
      \fill (\y-) circle[radius=2pt] node[below=1mm] {$\csname\y\endcsname$};
    };
    \def\0#1#2#3#4 {\draw (#2-) -- (#1+)
      (#3-) -- (#1+)
      pic[#4,draw,angle radius=6mm] {angle={#2-}--{#1+}--{#3-}};
    }
    \0{ozz}{zoz}{zzo}{<-}
    \0{zoz}{ozz}{ooz}{<-}
    \0{ooz}{ozz}{zoz}{<-}
    \0{zzo}{ozo}{zoo}{->}
    \0{ozo}{zzo}{zoo}{->}
    \0{zoo}{ooz}{ozo}{->}
  \end{tikzpicture}
\]
The diagram $\mathscr V$ is constructed as follows: every
`$a\mapsto bc$' in the extension of $\zeta$ to $\Sigma$ is
written as two edges in this diagram, going upwards from $b$ to $a$
and from $c$ to $a$, with an arrow from the first to the second. We
will use these arrows later; for now, let us call ``minimal edge'' the
source of the arrow, and ``maximal edge'' its range.

Let $P(\mathscr V)$ denote the space of all upwards-going paths in
$\mathscr V^\N$, recorded as sequences of vertices
$z=(z_0,z_1,\dots)$. The topology on $P(\mathscr V)$ declares as open
sets all
$\mathcal O_{(y_0,\dots,y_n)}=\{z\in P(\mathscr V):z_i=y_i\forall i\le
n\}$, and is homeomorphic to the Cantor set. (Note that we'll later
encode paths by their edges, but here it makes no difference.)

\begin{prop}
  The map $\lambda$ is a homeomorphism $\Omega\to P(\mathscr V)$.
\end{prop}
\begin{proof}
  First, $\lambda(\Omega)\subseteq P(\mathscr V)$: indeed the diagram
  just says that if $\widetilde w$ contains a letter $b$ or $c$, then
  this letter must appear inside $\zeta(a)$ for some $a\in\Sigma$, so
  the sequence $\lambda(w)$ must follow a path in $\mathscr V^\N$.

  Clearly $\lambda$ is continuous, since the $n$th letter of
  $\lambda(w)$ only depends on a finite portion (of size $2^n$ around
  the origin) of $w$.

  Consider next a path $z=(z_0,z_1,\dots)\in P(\mathscr V)$; we wish
  to show that it has precisely one preimage under $\lambda$. Define
  $\epsilon_i$ for $i\in\N$ as follows: if the edge $z_i\to z_{i+1}$
  in $z$ is minimal, set $\epsilon_i=0$, otherwise $\epsilon_i=1$. The
  preimage of the clopen $\mathcal O_{(z_0,\dots,z_n)}$ under
  $\lambda$ is the set of sequences $w\in\Sigma_2^\Z$ that contain
  $\zeta^n(z_n)$ at positions
  $[-\sum_{i<n}\epsilon_i2^i,2^n-\sum_{i<n}\epsilon_i2^i[$, and in 
  particular is non-empty; thus the intersection of these clopens is
  non-empty and $\lambda$ is surjective.

  If furthermore the sequence $(\epsilon_0,\epsilon_1,\dots,)$ is not
  eventually constant, then the intervals above grow left and right
  with union $\Z$, so the intersection of the above clopens is a single
  point and $\lambda$ is injective.

  It remains to consider the case of $(\epsilon_i)$ eventually
  constant, and easily reduce to the case of constant
  $\epsilon_i$. For future use, call \emph{minimal}, respectively
  \emph{maximal}, a path $z=(z_0,z_1,\dots)$ in $\mathscr V^\N$, if
  all its edges are so. It is easy to see that there are four maximal
  and four minimal infinite paths: dashed=minimal, solid=maximal:
  
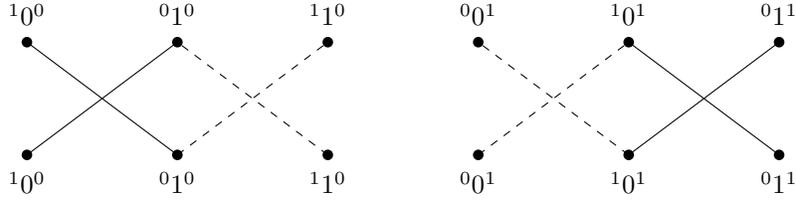
\[\begin{tikzpicture}
      \foreach\x/\y\z in {0/ozz,1/zoz,2/ooz,3/zzo,4/ozo,5/zoo} {
        \coordinate (\y-) at (2*\x,0);
        \coordinate (\y+) at (2*\x,1.5);
        \fill (\y+) circle[radius=2pt] node[above=1mm] {$\csname\y\endcsname$};
        \fill (\y-) circle[radius=2pt] node[below=1mm] {$\csname\y\endcsname$};
      };
      \draw (ozz-) -- (zoz+) (zoz-) -- (ozz+) (zoo-) -- (ozo+) (ozo-) -- (zoo+);
      \draw[dashed] (zoz-) -- (ooz+) (ooz-) -- (zoz+) (ozo-) -- (zzo+) (zzo-) -- (ozo+);
    \end{tikzpicture}
  \]
  The four minimal paths encode sequences as follows:
  \begin{align*}
    (\zoz,\ooz,\zoz,\dots) &= \lambda(\zeta^\infty(0).\zeta^\infty(1)),\\
    (\ooz,\zoz,\ooz,\dots) &= \lambda(\zeta^\infty(1).\zeta^\infty(1)),\\
    (\zzo,\ozo,\zzo,\dots) &= \lambda(\zeta^\infty(0).\zeta^\infty(0)),\\
    (\ozo,\zzo,\ozo,\dots) &= \lambda(\zeta^\infty(1).\zeta^\infty(0)),
  \end{align*}
  and the four maximal paths encode the same sequences, shifted one
  step left. It is easy to see that, in this case too, the map
  $\lambda$ is injective. Note that we needed the ``collaring'' here:
  without it, there would be only two encodings of minimal, or
  maximal, paths.
\end{proof}

The action of the semigroup $B(1,2)_+$ on $\Omega$ can now be
transported via $\lambda$ to $P(\mathscr V)$. Let us describe on
$P(\mathscr V)$ the action of $\sigma\colon\Omega\righttoleftarrow$;
it will be an ``adic'' transformation $\mu$.

Clearly for every vertex in $\mathscr V^\N$ there is a unique minimal
path ending at that vertex, namely the one defined going downwards by
always following the minimal edge.

For the path $(z_0,z_1,\dots)$: let $n$ be smallest such that the edge
$(z_{n-1},z_n)$ is not maximal; let $(z'_{n-1},z_n)$ be the corresponding maximal edge. Then define
$\mu(z_0,z_1,\dots)=(z'_0,\dots,z'_{n-1},z_n,\dots)$ where
$(z'_0,\dots,z'_{n-1})$ is the minimal path ending in $z'_{n-1}$.

In case there is no such $n$, this means that $(z_0,z_1,\dots)$ is
maximal, and we have to extend $\mu$ appropriately. In fact, there is
a unique continuous extension, as a limit of
$\mu(z_0,\dots,z_n,z'_{n+1},\dots)$ for non-maximal edges
$(z_n,z'_{n+1})$. This may be seen as follows: the first $n-1$
vertices of $\mu(z_0,\dots,z_n,z'_{n+1},\dots)$ are independent of the
choices at positions $>n$, as the following picture shows:
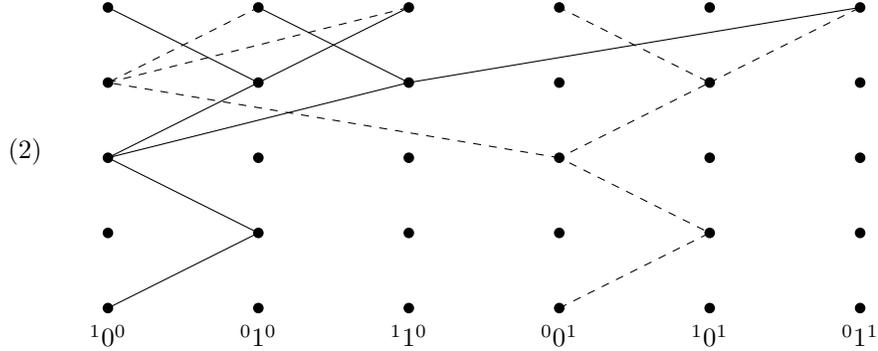
\begin{equation}\label{eq:minmax}
  \begin{tikzpicture}[baseline=4cm]
    \foreach\x/\y\z in {0/ozz,1/zoz,2/ooz,3/zzo,4/ozo,5/zoo} {
      \coordinate (\y2) at (2*\x,2);
      \foreach\t in {3,...,6} {
        \coordinate (\y\t) at (2*\x,\t);
        \fill (\y\t) circle[radius=2pt];
      }
      \fill (\y2) circle[radius=2pt] node[below=1mm] {$\csname\y\endcsname$};
    };
    \draw (ozz2) -- (zoz3) -- (ozz4) -- (zoz5) -- (ozz6)
    (zoz5) -- (ooz6)
    (ozz4) -- (ooz5) -- (zoz6)
    (ooz5) -- (zoo6);
    \draw [dashed] (zzo2) -- (ozo3) -- (zzo4) -- (ozo5) -- (zzo6)
    (ozo5) -- (zoo6)
    (zzo4) -- (ozz5) -- (zoz6)
    (ozz5) -- (ooz6);
  \end{tikzpicture}
\end{equation}
Here the path $(\ozz\;\zoz)^\omega$ is maximal, and I drew solid paths
coinciding with it on its first two or three edges; every non-maximal
path has to bifurcate away from it as indicated. The corresponding
successors are drawn in dashed, and they all start by a corresponding
prefix of $(\zzo\;\ozo)^\omega$.

Note therefore that the action of $\mu$ is almost finitary: except
for the four maximal paths, the paths $s$ and $\mu(s)$ are cofinal,
namely coincide starting from some point on.

Note also that it is here that we need the ``collaring'': without
having decorated letters with their left and right neighbours, we
couldn't know which maximal path goes to which minimal one.

The map $\zeta\colon\Omega\righttoleftarrow$ is transported via
$\lambda$ to a section $\sigma^{-1}_0$ of the shift $\sigma$ on
$P(\mathscr V)$, and simply prepends to a path in $P(\mathscr V)$ the
unique minimal edge abutting to its start vertex.

\begin{prop}
  The map $\alpha\mapsto\mu,\zeta\mapsto\sigma^{-1}_0$ defines an
  action of $B(1,2)_+$ on $P(\mathscr V)$, and turns $\lambda$ into an
  equivariant homeomorphism.
\end{prop}
\begin{proof}
  The claim is immediate, except perhaps for why
  $\sigma\colon\Omega\righttoleftarrow$ is transported to $\mu$. Now
  if $w\in\zeta(\Omega)$ then $\lambda(w)$ starts with a minimal edge
  and $\mu\lambda(w)$ coincides with $\lambda(w)$ except that its
  first edge is now maximal; thus
  $\lambda^{-1}\mu\lambda(w)\in\sigma\zeta(\Omega)$, and is in fact
  $\sigma(w)$, so $\mu\circ\lambda=\lambda\circ\sigma$ in that
  case. If $w\in\sigma\zeta(\Omega)$ then $\sigma(w)\in\zeta(\Omega)$
  and $\zeta^{-1}\sigma(w)=\sigma\zeta^{-1}(w)$; and similarly
  $\lambda(w)$ starts with a maximal edge and $\mu\lambda(w)=z_0 z'$
  with $z_0$ minimal and $z'=\mu(\sigma\lambda(w))$, as required.
\end{proof}

Note that we may easily define an inverse of $\lambda$ on
$P(\mathscr V)$ using $\mu$: we have
\[\lambda^{-1}(z)=(\mu^n(z)_0)_{n\in\Z}.\]

\section{Automatic actions and $\omega$-regular languages}
We rephrase the previous section in the more flexible language of
automata.  We first recall the notion of $\omega$-regular
languages\footnote{In this section we switch from the notation
  $\Sigma^\N$ to $\Sigma^\omega$ out of deference for this standard
  terminology. We will switch back to $\N$ when we embed $\Sigma^\N$
  into $\Sigma^\Z$ in the next section.}. We fix once and for all a
finite alphabet $\Sigma$. An \emph{$\omega$-automaton} is the data of
a finite directed graph $\mathscr A$, two subsets $A_*,A_\dagger$ of
its vertex set called \emph{initial} and \emph{final} states, and a
labelling of edges of $\mathscr A$ by $\Sigma$. The
\emph{$\omega$-language} that it recognizes is the following subset
$L(\mathscr A)$ of the set $\Sigma^\omega$ of right-infinite words
over $\Sigma$: it consists of those $w\in\Sigma^\omega$ for which
there exists a path in $\mathscr A$ labeled $w$, starting at a vertex
in $A_*$, and passing infinitely many times through vertices in
$A_\dagger$. A subset $L\subseteq\Sigma^\omega$ that can be recognized
by an $\omega$-automaton is called an \emph{$\omega$-regular
  language}.

For example, $P(\mathscr V)$ is recognized by the following
$\omega$-automaton in which all states are initial and final; it is
obtained by identifying the top and bottom rows in $\mathscr V$:
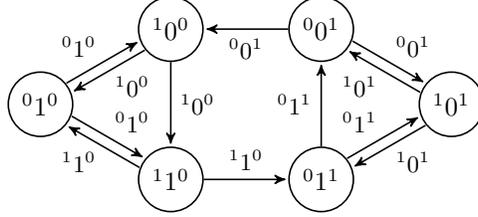
\[\begin{fsa}
    \node[state] (ozz) at (-1,1) {$\ozz$};
    \node[state] (zzo) at (1,1) {$\zzo$};
    \node[state] (ooz) at (-1,-1) {$\ooz$};
    \node[state] (zoo) at (1,-1) {$\zoo$};
    \node[state] (zoz) at (-2.732,0) {$\zoz$};
    \node[state] (ozo) at (2.732,0) {$\ozo$};
    \foreach\i/\j in {ozz/ooz,ooz/zoo,zoo/zzo,zzo/ozz} {
      \draw[->] (\i) -- node {\small $\csname\i\endcsname$} (\j);
    };
    \foreach\i/\j in {ozz/zoz,zoz/ozz,zoz/ooz,ooz/zoz,zzo/ozo,ozo/zzo,ozo/zoo,zoo/ozo} {
      \draw[->,decorate,decoration={single line,raise=2pt}] (\i) -- node {\small $\csname\i\endcsname$} (\j);
    };
  \end{fsa}
\]
Note that each edge's label is simply the label of its source vertex.

\begin{defn}
  Let $L$ be an $\omega$-regular language, and let a (semi)group $G$ acting
  on $L\subseteq\Sigma^\omega$. The action is called \emph{regular} if
  for every $g\in G$ its graph
  \[\{(w,g(w)):w\in L\}\subseteq L\times L
  \]
  is a regular language in
  $\Sigma^\omega\times\Sigma^\omega=(\Sigma\times\Sigma)^\omega$.
\end{defn}

By classical properties of $\omega$-regular languages, it is
sufficient to check that the graphs of generators are regular.

We shall see that the action of $\mu$ on $P(\mathscr V)$ is regular. However, before doing so, we change once more the notation: first, we write $L(\mathscr A)$ instead of $P(\mathscr V)$, since we are about to forget about the Bratteli diagram. We change our alphabet to $\{i_j:i\in\{0,1\},j\in\{a,b,c,d,e,f\}\}$ the set of edges of $\mathscr A$, as follows: we rename vertices as
\[\ozz=a,\quad\zoz=b,\quad\ooz=c,\quad \zzo=d,\quad\ozo=e,\quad\zoo=f\]
so as to avoid multiple subscripts, and label the minimal edge ending at vertex $j$ as $0_j$ and the maximal one as $1_j$. We thus get
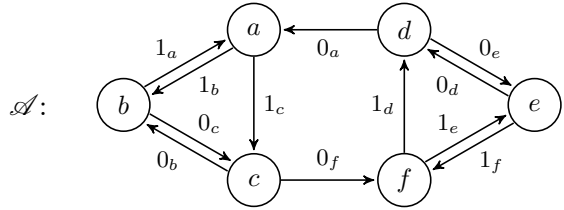
\[\begin{fsa}[every state/.style={minimum size=7mm}]
    \node at (-4,0) {\large $\mathscr A$:};
    \node[state] (ozz) at (-1,1) {$a$};
    \node[state] (zzo) at (1,1) {$d$};
    \node[state] (ooz) at (-1,-1) {$c$};
    \node[state] (zoo) at (1,-1) {$f$};
    \node[state] (zoz) at (-2.732,0) {$b$};
    \node[state] (ozo) at (2.732,0) {$e$};
    \foreach\i/\j/\l in {ozz/ooz/1_c,ooz/zoo/0_f,zoo/zzo/1_d,zzo/ozz/0_a} {
      \draw[->] (\i) -- node {\small $\l$} (\j);
    };
    \foreach\i/\j\l in {ozz/zoz/1_b,zoz/ozz/1_a,zoz/ooz/0_c,ooz/zoz/0_b,zzo/ozo/0_e,ozo/zzo/0_d,ozo/zoo/1_f,zoo/ozo/1_e} {
      \draw[->,decorate,decoration={single line,raise=2pt}] (\i) -- node {\small $\l$} (\j);
    };
  \end{fsa}
\]
\begin{lem}
  The $\omega$-language accepted by $\mathscr A$ is $P(\mathscr V)$.
\end{lem}
\begin{proof}
  The translation is direct: the states of $\mathscr A$ are in
  bijection with the vertices of $\mathscr V$, and every edge `$v\to w$'
  in $\mathscr V$ gives a transition in $\mathscr A$ from $v$ to $w$,
  labeled `$0_w$' or `$1_w$' according to whether the edge is minimal
  or maximal.
\end{proof}

Let us now construct an automaton recognizing the action of $\mu$. We
shall in fact add to $\mathscr A$ new states $\mu_{jk}$ for all states
$j,k$ of $\mathscr A$, representing the action of $\mu$ on paths
starting at $j$ when their image starts at $k$. The labels of the
edges are written `$e|f$' rather than $(e,f)$ for $e,f$ in our
alphabet $\{i_j\}$; and even `$e$' rather than $(e,e)$.

Our new automaton will in fact contain the previous one; choosing as
initial states all $j\in\{a,b,c,d,e,f\}$ yields $L(\mathscr A)$, and
in other words the action of the identity if we identify
$L(\mathscr A)$ with its diagonal in
$L(\mathscr A)\times L(\mathscr A)$. Choosing as initial states all
states labeled $\mu_{jk}$ yields the action of $\mu$. All states in
$\mathscr A_\mu$ are final. We recall:
\begin{prop}[e.g.~\cite{perrin-pin:iw}*{Proposition~3.7}]\label{prop:closed}
  An $\omega$-regular language $L$ is closed in $\Sigma^\omega$ if and
  only $L=L(\mathscr A)$ for an automaton $\mathscr A$ in which all
  states are final.
\end{prop}
Thus we recover that $L(\mathscr A)$ is compact, and the
transformation defined by $\mathscr A_\mu$ is continuous.

The information in the Bratteli diagram can be translated as follows:
if there are minimal and maximal edges $j\to\ell$ and $k\to\ell$
respectively in the Bratteli diagram, then the automaton has a
transition from $\mu_{jk}$ to $\ell$ labeled $0_\ell\to1_\ell$. It
also has some transitions from $\mu_{k m}$ to $\mu_{\ell n}$, but
these are not entirely prescribed by the Bratteli diagram (remember
the required argument about continuity!). In fact, to determine such
edges, we must consider in $\mathscr A$ all possible continuations of
the path starting with the edge $1_k$, once it reaches an edge of the form $0_p$
replace it by $1_p$, and follow backwards the $0_q$-edges. Thus for
example there is a path $1_a1_b\dots 1_a1_b0_c$, which goes under
$\mu$ to $0_d0_e\dots0_d0_a1_c$, so there is an edge from $\mu_{be}$
to $\mu_{ad}$ labeled $1_a|0_d$. Writing labels `$i_j$' for `$i_j|i_j$' to highlight the previous automaton as a subautomaton, we get
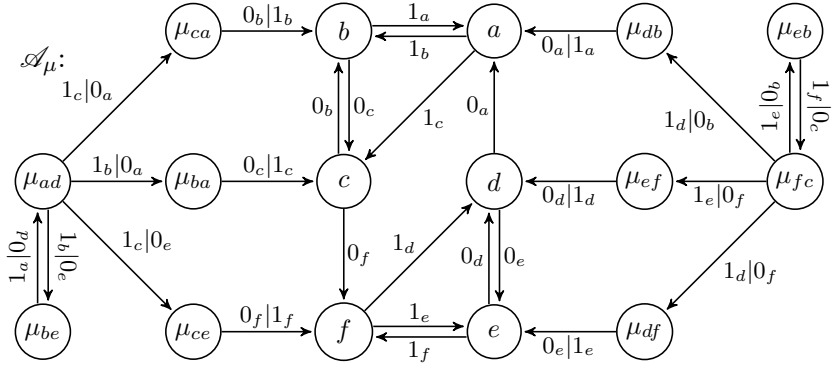
\[\begin{fsa}[every state/.style={minimum size=7mm,text width=7mm,inner sep=0mm,align=center}]
    \node at (-5,1.6) {\large $\mathscr A_\mu$:};
    \node[state] (a) at (1,2) {$a$};
    \node[state] (d) at (1,0) {$d$};
    \node[state] (c) at (-1,0) {$c$};
    \node[state] (f) at (-1,-2) {$f$};
    \node[state] (b) at (-1,2) {$b$};
    \node[state] (e) at (1,-2) {$e$};
    \node[state] (muca) at (-3,2) {$\mu_{ca}$};
    \node[state] (muba) at (-3,0) {$\mu_{ba}$};
    \node[state] (muce) at (-3,-2) {$\mu_{ce}$};
    \node[state] (muad) at (-5,0) {$\mu_{ad}$};
    \node[state] (mube) at (-5,-2) {$\mu_{be}$};
    \node[state] (mudf) at (3,-2) {$\mu_{df}$};
    \node[state] (muef) at (3,0) {$\mu_{ef}$};
    \node[state] (mudb) at (3,2) {$\mu_{db}$};
    \node[state] (mufc) at (5,0) {$\mu_{fc}$};
    \node[state] (mueb) at (5,2) {$\mu_{eb}$};
    \foreach\i/\j/\l in {a/c/1_c,c/f/0_f,f/d/1_d,d/a/0_a,muca/b/0_b|1_b,muba/c/0_c|1_c,muce/f/0_f|1_f,mudf/e/0_e|1_e,muef/d/0_d|1_d,mudb/a/0_a|1_a,muad/muca/1_c|0_a,muad/muba/1_b|0_a,muad/muce/1_c|0_e,mufc/mudb/1_d|0_b,mufc/muef/1_e|0_f,mufc/mudf/1_d|0_f} {
      \draw[->] (\i) -- node [inner sep=1pt] {\small $\l$} (\j);
    };
    \foreach\i/\j\l in {a/b/1_b,b/a/1_a,b/c/0_c,c/b/0_b,d/e/0_e,e/d/0_d,e/f/1_f,f/e/1_e} {
      \draw[->,decorate,decoration={single line,raise=2pt}] (\i) -- node {\small $\l$} (\j);
    };
    \foreach\i/\j\l in {muad/mube/1_b|0_e,mueb/mufc/1_f|0_c,mube/muad/1_a|0_d,mufc/mueb/1_e|0_b} {
      \draw[->,decorate,decoration={single line,raise=2pt}] (\i) -- node [sloped,midway,above=1pt] {\small $\l$} (\j);
    };    
  \end{fsa}
\]
The initial states are all those labeled `$\mu_{jk}$', and all states
are final. Note that the transducer describing $\mu$ is ``bounded'' in
the following sense: for every $n\in\N$, there is a bounded number (at
most $12$) paths of length $n$ that do not reach an identity state
($a,\dots,f$).

!!Also for economy, we omit from the automaton all states that are not
\emph{accessible} (cannot be reached from an initial state) or not
\emph{co-accessible} (cannot be followed by a path traversing final
states infinitely often). Automata can be \emph{minimized} by
furthermore identifying indistinguishable states. The minimal
automaton associated with an $\omega$-regular language is unique.

Of course the automaton describing $\mu^{-1}$ is obtained by changing
the labels `$e|f$' to `$f|e$'. Automata for $\mu^n$, with arbitrary
$n\in\Z$, may be obtained by composing the transducers in the usual
way.
\begin{lem}
  The automaton $\mathscr A_\mu$ defines the homeomorphism $\mu$ on
  $L(\mathscr A)$.
\end{lem}
\begin{proof}
  First, let us check that the relation defined by $\mathscr A_\mu$ is
  a homeomorphism. For this, just keep the input labels on each edge
  of $\mathscr A_\mu$, and note that the resulting automaton
  unambiguously minimizes to $\mathscr A$; and similarly when only
  keeping the output labels.

  Next, note that the automaton really does the following: while a
  maximal edge is read, print a minimal one and repeat. The minimal
  edge to be printed follows from the computation in
  diagram~\eqref{eq:minmax}.
\end{proof}

Note that transducers may also describe non-invertible
transformations, and even relations. For example, the map $\zeta$,
which prepends to a path its minimal edge, is defined by the following
automaton $\mathscr A_\zeta$ (changing the initial states leads to
$\sigma\zeta$, the map prepending to each path its maximal edge; and
switching input and output leads to the shift map $\sigma$ on paths);
so the action of the semigroup $B(1,2)_+$ on $L(\mathscr A)$ is
automatic. The automaton has stateset the alphabet, and on each
transition reads a letter, printing the previously-stored one; so if
$\mathscr A$ has a transition from $i$ to $j$ labeled `$z_j$' then
$\mathscr A_\zeta$ has for all $y\in\{0,1\}$ a transition from $y_i$
to $z_j$ labeled `$z_j|y_i$':
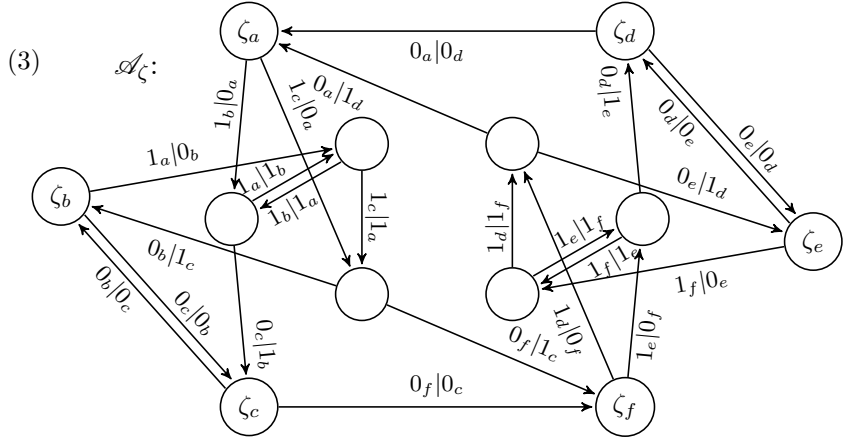
\begin{equation}\label{eq:zeta}
  \hbox to 100mm{\kern-1cm\begin{fsa}[every state/.style={minimum size=7mm,text width=7mm,inner sep=0mm,align=center},baseline=2cm]
    \node at (-4,2) {\large $\mathscr A_\zeta$:};
    \node[state] (a1) at (-1,1) {};
    \node[state] (d1) at (1,1) {};
    \node[state] (c1) at (-1,-1) {};
    \node[state] (f1) at (1,-1) {};
    \node[state] (b1) at (-2.732,0) {};
    \node[state] (e1) at (2.732,0) {};
    \node[state] (a0) at (-2.5,2.5) {$\zeta_a$};
    \node[state] (d0) at (2.5,2.5) {$\zeta_d$};
    \node[state] (c0) at (-2.5,-2.5) {$\zeta_c$};
    \node[state] (f0) at (2.5,-2.5) {$\zeta_f$};
    \node[state] (b0) at (-5,0.3) {$\zeta_b$};
    \node[state] (e0) at (5,-0.3) {$\zeta_e$};
    \foreach\i/\j/\k/\l in {c/f/0/0,d/a/0/0,c/f/1/0,d/a/1/0,b/a/0/1,b/c/1/0,c/b/1/0,d/e/1/0,e/d/1/0,e/f/0/1} {
      \draw[->] (\i\k) -- node[midway,sloped] {\small $\l_\j|\k_\i$} (\j\l);
    };
    \foreach\i/\j/\k/\l/\s in {a/b/0/1/{(-2mm,5mm)},f/e/0/1/{(2mm,-5mm)},a/c/1/1/{(-5mm,2mm)},f/d/1/1/{(4mm,2mm)},a/c/0/1/{(-12mm,0mm)},f/d/0/1/{(12mm,0mm)}} {
      \draw[->] (\i\k) -- node[midway,sloped,shift=\s] {\small $\l_\j|\k_\i$} (\j\l);
    };
    \foreach\i/\j\l/\s in {a/b/1/{(-7mm,0mm)},b/a/1/{(1mm,0mm)},b/c/0/{(0mm,0mm)},c/b/0/{(0mm,0mm)},d/e/0/{(0mm,0mm)},e/d/0/{(0mm,0mm)},e/f/1/{(-1mm,0mm)},f/e/1/{(7mm,0mm)}} {
      \draw[->,decorate,decoration={single line,raise=2pt}] (\i\l) -- node[midway,sloped,shift=\s] {\small $\l_\j|\l_\i$} (\j\l);
    };
  \end{fsa}}
\end{equation}

It is instructive to compare the above automaton with the adic
transformation given in~\cite{vershik-solomyak:adic}*{Equation~(2)}:
in our language, it is given by the transducer
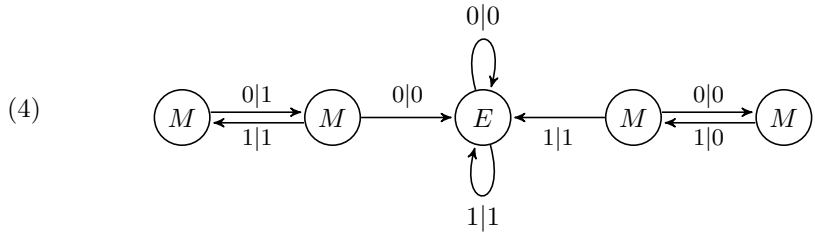
\begin{equation}\label{eq:M}
  \begin{fsa}[every state/.style={minimum size=7mm,text width=7mm,inner sep=0mm,align=center},baseline=0mm]
    \node[state] (1) at (0,0) {$E$};
    \node[state] (01e) at (4,0) {$M$};
    \node[state] (01o) at (-2,0) {$M$};
    \node[state] (10e) at (2,0) {$M$};
    \node[state] (10o) at (-4,0) {$M$};
    \path (1) edge[->,loop above] node {$0|0$} (1);
    \path (1) edge[->,loop below] node {$1|1$} (1);
    \foreach\i/\j\l in {01e/10e/1|0,10e/01e/0|0,01o/10o/1|1,10o/01o/0|1} {
      \draw[->,decorate,decoration={single line,raise=2pt}] (\i) -- node {\small $\l$} (\j);
    };
    \foreach\i/\j\l in {10e/1/1|1,01o/1/0|0} {
      \draw[->] (\i) -- node {\small $\l$} (\j);
    };
  \end{fsa}
\end{equation}
It does not define a continuous self-map of $\{0,1\}^\omega$;
according to taste, $M$ may be considered to be a discontinuous
self-map (with discontinuity locus $\{(01)^\omega,(10)^\omega\}$); or
a relation that it two-valued at these points and defines a map
elsewhere; or a map that is well-defined an continuous on
$L\coloneqq\{0,1\}^\omega\setminus\{(01)^\omega,(10)^\omega\}$, if one
declares only the identity state $E$ to be final. See below for the
connection between $\mu$ and $M$.

\section{The covering map to $2$-adics}
The substitution $\zeta$ has fixed length (all images of letters have
length $2$), or more pedantically said the substitution $\zeta$
factors via $\{0,1\}\mapsto\{\bullet\}$ to the $1$-letter substitution
$\bullet\mapsto\bullet\bullet$. Its Bratteli-Vershik diagram has one
vertex and two edges; the corresponding automaton is
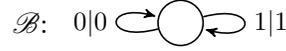
\[\begin{fsa}
    \node at (-2,0) {\large $\mathscr B$:};
    \node[state,minimum size=6mm] (L) {};
    \path (L) edge[loop left] node {$0|0$} (L) edge[loop right] node {$1|1$} (L);
  \end{fsa}
\]
The automorphism $\tau$ of $\{0,1\}^\omega$ given by Bratteli-Vershik dynamics is the odometer; adding it to the automaton above, we get
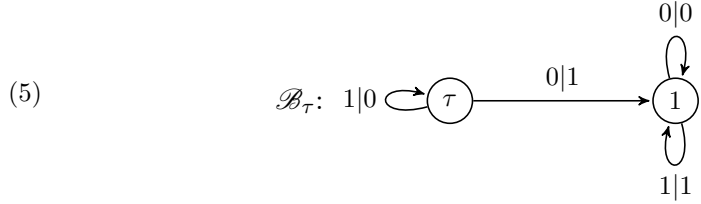
\begin{equation}\label{eq:tau}
  \begin{fsa}[baseline=0mm]
    \node at (-2,0) {\large $\mathscr B_\tau$:};
    \node[state,minimum size=6mm] (tau) at (0,0) {$\tau$};
    \node[state,minimum size=6mm] (L) at (3,0) {$1$};
    \path (tau) edge[loop left] node {$1|0$} (tau) edge node {$0|1$} (L);
    \path (L) edge[loop above] node {$0|0$} (L) edge[loop below] node {$1|1$} (L);
  \end{fsa}
\end{equation}
Identifying $2$-adic numbers with the sequence of their digits, one
immediately sees that the $\omega$-language $L(\mathscr B)$ is
identified with $\Z_2$, and under this identification
$\tau(z)=z+1$. Define also on $\Z_2$ the doubling map $\zeta(z)=2z$,
and note that $(\alpha\mapsto\tau,\zeta\mapsto\zeta)$ gives an action
of the semigroup $B(1,2)_+$ on $\Z_2$, as a group of affine maps. The
automaton giving $\zeta$ first prints a `$0$', memorizing the
just-read letter, and then copies its input delayed one time unit
(compare with~\eqref{eq:zeta}):
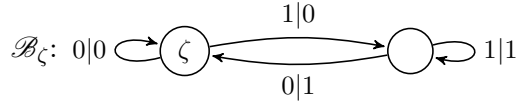
\[\begin{fsa}
    \node at (-2,0) {\large $\mathscr B_\zeta$:};
    \node[state,minimum size=6mm] (z0) at (0,0) {$\zeta$};
    \node[state,minimum size=6mm] (z1) at (3,0) {};
    \path (z0) edge[bend left=10] node {$1|0$} (z1)
    (z0) edge[loop left] node {$0|0$} (z0)
    (z1) edge[bend left=10] node {$0|1$} (z0)
    (z1) edge[loop right] node {$1|1$} (z1);
  \end{fsa}
\]

\begin{thm}\label{thm:factor}
  The Thue-Morse system factors onto the odometer: there is a
  continuous map $\pi\colon L(\mathscr A)\twoheadrightarrow\Z_2$ that
  interlaces the actions of $B(1,2)_+$.
\end{thm}
\begin{proof}
  This is immediately checked on the automata: if one forgets letter
  decorations by replacing letter `$i_j$' by `$i$', there exists a
  ``morphism of automata'', namely an initial-state-, final-state- and
  label-preserving graph morphism, from $\mathscr A$ to $\mathscr B$
  and from $\mathscr A_\mu$ to $\mathscr B_\tau$; so
  $\tau\circ\pi=\pi\circ\mu$. The same statement may be checked for
  $\zeta$, either via automata or directly.
\end{proof}

Much can be said about that factor map $\pi$. First, fibres typically
have cardinality $2$; though $0^\omega$ and $1^\omega$ have $4$
preimages. Note that these are precisely the elements of $\Omega$ with
an additional symmetry, i.e.\ on which $D_\infty$ does not act freely.

\def\8{\texttt0}\def\9{\texttt1} Let us be more precise. First, there
is an order-$2$ symmetry in $L(\mathscr A)$, given on its automaton by
a half-turn. One might therefore consider rather an intermediate
quotient between $L(\mathscr A)$ and $\Z_2$, of the form
$\Z/2\cdot\Z_2$; the map $L(\mathscr A)\to\Z/2\cdot\Z_2$ is given by
$z\mapsto(s(z),\pi(z))$ with $s(z)=0$ if $z$ starts in $\{a,d,e\}$ and
$s(z)=1$ if $z$ starts in $\{b,c,f\}$. (Going back to the
interpretation of states as letters ${}^xy^z$, we use for $s(z)$ the
letter $y$ of the initial vertex of $z$). There remains the issue of
eventually-minimal and -maximal paths; those are the paths mapping
under $\pi$ to $\Z\subset\Z_2$. We thus replace $\Z_2$ by
\[\widetilde\Z_2\coloneqq(\Z_2\setminus\Z)\sqcup(\Z\times\Z/2),
\]
a topological space with a free action of $\Z$ in which the topology
is defined by declaring as open neighbourhoods of
$(0,t)\in\Z\times\Z/2$ all
$\bigsqcup_{n\ge N}2^{2n+t}(2\Z_2+1)\sqcup\{(0,t)\}$ for $N\in\N$, and
of course their $\Z$-translates. In other words, numbers with binary
representation $\8^n\texttt{1*}\in\Z_2$ are close to
$(\8^n\texttt{1*},n\bmod 2)\in\Z\times\Z/2$. There is an obvious map
$\widetilde\Z_2\to\Z_2$ given by the identity on $\Z_2\setminus\Z$ and
$(n,t)\mapsto n$ on $\Z\times\Z/2$.

The $\Z$-action is given on $\widetilde\Z_2$ by
$\widetilde\tau(z)=z+1$ and $\widetilde\tau(z,t)=(z+1,t)$. The
non-invertible dynamics $\tau\colon\Z_2\righttoleftarrow$ also lift to
$\widetilde\Z_2$: one defines a map $\widetilde\zeta$ on
$\widetilde\Z_2$ by $\widetilde\zeta(z)=2z$ and
$\widetilde\zeta(z,t)=(2z,t+1)$. We thus have an action of $B(1,2)_+$
on $\widetilde\Z_2$, compatible with the map $\widetilde\Z_2\to\Z_2$.

We can then improve the map $\pi$ into a map
$\widetilde\pi\colon L(\mathscr A)\to\widetilde\Z_2$ as follows: if
$z\in L(\mathscr A)$ is path not ending in $0^\omega$ or $1^\omega$,
then $\widetilde\pi(z)$ is the element $\pi(z)$ of $\Z_2\setminus\Z$
read along the labels of the path. If however $z$ ends in $0^\omega$
or $1^\omega$, then $\widetilde\pi(z)$ is $(\pi(z),t)$ with $t=0$ if
the path is in $\{b,e\}$ at arbitrarily large even positions, and
$t=1$ if the path is in $\{b,e\}$ at arbitrarily large odd positions.

\begin{thm}\label{thm:homeo}
  The map $z\mapsto(\widetilde\pi(z),s(z))$ is a homeomorphism between
  $L(\mathscr A)$ and $\widetilde\Z_2\times\Z/2$, and the
  homeomorphism $\mu$ translates, via this homeomorphism, to
  ``addition with a cocycle''
  \[\widetilde\mu\colon(z,s)\mapsto(z+1,s+\phi(z+1));\]
  The cocycle $\phi\colon\widetilde\Z_2\to\Z/2$ is given by
  \begin{equation}\label{eq:phi}
    \phi(2^n(2\Z_2+1))=\phi(2^n(2\Z+1),t)=n+1\bmod 2,\qquad \phi(0,t)=t+1.
  \end{equation}
\end{thm}
In terms of binary expansions, the cocycle is given by
$\phi(\8^n\texttt{1*})=\phi(\8^n\texttt{1*},t)=n+1\bmod 2$ and
$\phi(\8^\omega,t)=t+1$. Let us state a fundamental property of
$\mathscr A$, which comes from an analogous statement for the
substitution $\zeta$ and which we have already implicitly used:
\newtheorem{lemdef}[thm]{Lemma-Definition}
\begin{lemdef}\label{ld:p}
  Given a sequence $w\in\{0,1\}^{-\N}$ and a state $h\in\mathscr A$,
  there is a unique reverse path $p_{w,h}\colon-\N\to\mathscr A$
  ending in $h$ and whose labels project to $w$ under forgetting
  decorations.\qed
\end{lemdef}

\begin{proof}[Proof of Theorem~\ref{thm:homeo}]
  First, the map $(\widetilde\pi,s)$ is continuous: on paths not
  ending in $0^\omega$ or $1^\omega$, the first $n$ bits of its output
  depend only on the first $n$ edges. For paths ending in $0^\omega$
  or $1^\omega$, the paths crossing $\{b,e\}$ at even, respectively
  odd positions are in disjoint clopens.

  \def\8{\texttt0}\def\9{\texttt1}
  We turn to bijectivity of $\widetilde\mu$.  Given a ``symmetry'' bit
  $s$ and a bit sequence $x=x_0x_1\dots$ representing an element of
  $\widetilde\Z$ (with an additional bit in case the sequence is
  ultimately constant): consider for large $m\in\N$ the six reverse
  paths $p_{0^{-\omega}x_0\dots x_m,h}$ given by Lemma-Definition~\ref{ld:p}.
  If $x$ starts by $\8^{2n+1}\9$ then these paths must start in
  $\{c,d\}$ and the symmetry bit decides which; if $x$ starts in
  $\8^{2n}\9$ a similar (slightly more complicated) reasoning holds;
  if $x\in\{\8^\omega,\9^\omega\}$ then the additional bit determines
  where the preimage path starts. In all cases, the starting state of
  the preimage of $(x,s)$ is uniquely determined, and its successive
  states are determined by induction.

  We next show that $\mu$ is carried to $\widetilde\mu$. Without loss
  of generality: consider a path $p$ starting at some vertex in
  $\{a,d,e\}$, i.e.\ with $s(z)=0$. If $p$ starts with labels
  ${}_e1_*^{2n}0_*$ then $\mu$ maps it to a path starting at $b$; if
  $p$ starts with ${}_a1_*^{2n+1}0_*$ then $\mu$ maps it to a path
  starting at $d$; and if $p$ starts with label ${}_d0_*$ then $\mu$
  maps it to a path starting at $b$ or $f$; so if $p$ starts with
  $1^{2n}0$ then the symmetry bit $s$ must be flipped while if $p$
  starts with $1^{2n+1}0$ then $s$ should not be changed.

  Consider finally the action on $(x,t)\in\Z\times\Z/2$, again with
  ``symmetry'' bit $s=0$. If $x\neq-1$, then the same argument as
  above holds. If $x=-1$, then the path above $((x,t),0)$ is either
  $(1_a1_b)^\omega$ or $(1_e1_f)^\omega$. In the first case, $t=1$ and its
  image under $\mu$ is $(0_d0_e)^\omega$, in the second case, $t=0$ and
  its image under $\mu$ is $(0_b0_c)^\omega$.
\end{proof}

The connection between $\mu$ and the map $M$ from~\eqref{eq:M} is via
the ``difference operator'' $D\colon\{0,1\}^\omega\to\{0,1\}^\omega$,
given by
\[D(x_0,x_1,\dots)=(x_0+x_1\bmod 2,x_1+x_2\bmod 2,\dots);\]
it is a $2:1$ map implemented by the automaton
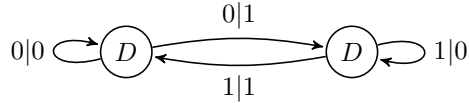
\[\begin{fsa}
    \node[state] (d0) at (0,0) {$D$};
    \node[state] (d1) at (3,0) {$D$};
    \path (d0) edge[->,loop left] node {$0|0$} (d0)
    edge[->,bend left=10] node {$0|1$} (d1);
    \path (d1) edge[->,loop right] node {$1|0$} (d1)
    edge[->,bend left=10] node {$1|1$} (d0);
  \end{fsa}
\]
and one immediately checks, by computing the corresponding transducers
and identifying $\Z_2$ with $\{0,1\}^\omega$, the identity
$\tau\circ D=D\circ M$.

\section{Natural extensions and limit spaces}
The dynamical system $(L(\mathscr A),\sigma)=(P(\mathscr V),\sigma)$
admits a \emph{natural extension}: a topological space $\mathfrak X$
equipped with a map $\mathfrak X\twoheadrightarrow L(\mathscr A)$ and
a self-homeomorphism inducing $\sigma$ on $L(\mathscr A)$, universal
for these properties. It may be constructed as
\[\widehat L(\mathscr A)\coloneqq\projlim(L(\mathscr A),\sigma),\]
namely the space of sequences
$(z_0,z_{-1},\dots)\in L(\mathscr A)^{-\N}$ such that
$\sigma_{L(\mathscr A)}(z_i)=z_{i+1}$ for all $i\le-1$. The one-sided
shift on $\widehat L(\mathscr A)$ is bijective, and there is a natural
map $\widehat L(\mathscr A)\to L(\mathscr A)$ given by
$(z_0,z_{-1},\dots)\mapsto z_0$ which interlaces the shifts
$\sigma_{\widehat L(\mathscr A)}$ and $\sigma_{L(\mathscr A)}$.

Recall that $L(\mathscr A)$ is a subset of $\Sigma^\N$ for
$\Sigma=\{0_a,1_a,\dots,0_f,1_f\}$, and $\sigma$ is as usual induced
by the shift on $\Sigma^\N$; in fact $L(\mathscr A)$ is the set of
$\omega$-paths in the graph $\mathscr A$. Thus
$\widehat L(\mathscr A)$ is naturally the set of two-sided infinite
paths in $\mathscr A$, a subset of $\Sigma^\Z$.

Consider the ``Baumslag-Solitar group''
\[B(1,2)=\langle\alpha,\zeta\mid\alpha^2\circ\zeta=\zeta\circ\alpha\rangle.\]
We have an action of $B(1,2)$ on $\widehat L(\mathscr A)$, defined as
follows: $\zeta$ is the inverse of the shift $\sigma$ on
$\widehat L(\mathscr A)\subset\Sigma^\Z$. Using the paths $p_{w,h}$
from Lemma-Definition~\ref{ld:p}, the homeomorphism $\mu$ of
$L(\mathscr A)$ induces a homeomorphism $\widehat\mu$ of
$\widehat L(\mathscr A)$ by
\[\widehat\mu(p_{w,h}.q)=p_{w,h'}.\mu(q)\text{ where $h'$ is the initial vertex of }\mu(q),\]
and we let $\alpha$ act as $\widehat\mu$.
\begin{lem}
  The above defines an action of $B(1,2)$ on $\widehat L(\mathscr A)$.
\end{lem}
\begin{proof}
  It suffices to check the relation. Since we will need it later, here
  is the automaton computing $\mu^2$ (only half of it is drawn, the other half is
  symmetric):
  
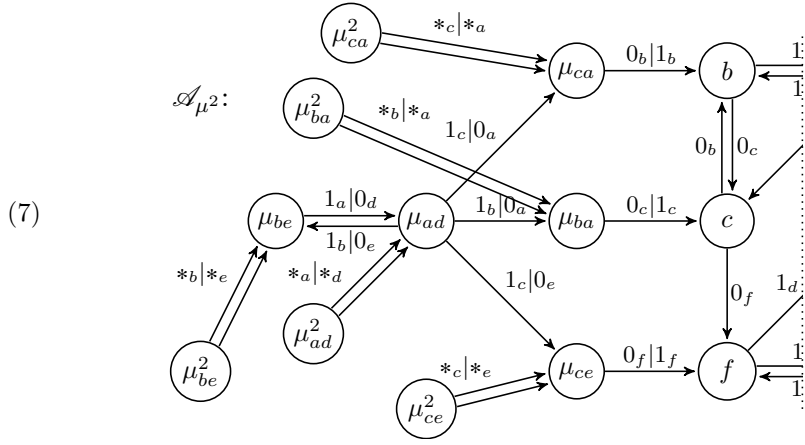
\begin{equation}\label{eq:mu2}
    \begin{fsa}[every state/.style={minimum size=7mm,text width=7mm,inner sep=0mm,align=center},baseline=0mm]
    \path[dotted,draw,-] (0,-2.5) -- (0,2.5);
    \clip (0,-2.9) rectangle (-8.5,3);
    \node at (-8,1.6) {\large $\mathscr A_{\mu^2}$:};
    \node[state] (a) at (1,2) {$a$};
    \node[state] (d) at (1,0) {$d$};
    \node[state] (c) at (-1,0) {$c$};
    \node[state] (f) at (-1,-2) {$f$};
    \node[state] (b) at (-1,2) {$b$};
    \node[state] (e) at (1,-2) {$e$};
    \node[state] (muca) at (-3,2) {$\mu_{ca}$};
    \node[state] (muba) at (-3,0) {$\mu_{ba}$};
    \node[state] (muce) at (-3,-2) {$\mu_{ce}$};
    \node[state] (muad) at (-5,0) {$\mu_{ad}$};
    \node[state] (mube) at (-7,0) {$\mu_{be}$};
    \node[state] (mudf) at (3,-2) {$\mu_{df}$};
    \node[state] (muef) at (3,0) {$\mu_{ef}$};
    \node[state] (mudb) at (3,2) {$\mu_{db}$};
    \node[state] (mufc) at (5,0) {$\mu_{fc}$};
    \node[state] (mueb) at (5,2) {$\mu_{eb}$};
    \node[state] (mu2ca) at (-6,2.5) {$\mu^2_{ca}$};
    \node[state] (mu2ba) at (-6.5,1.5) {$\mu^2_{ba}$};
    \node[state] (mu2ce) at (-5,-2.5) {$\mu^2_{ce}$};
    \node[state] (mu2ad) at (-6.5,-1.5) {$\mu^2_{ad}$};
    \node[state] (mu2be) at (-8,-2) {$\mu^2_{be}$};
    \foreach\i/\j/\l in {a/c/1_c,c/f/0_f,f/d/1_d,d/a/0_a,muca/b/0_b|1_b,muba/c/0_c|1_c,muce/f/0_f|1_f,mudf/e/0_e|1_e,muef/d/0_d|1_d,mudb/a/0_a|1_a,muad/muca/1_c|0_a,muad/muba/1_b|0_a,muad/muce/1_c|0_e,mufc/mudb/1_d|0_b,mufc/muef/1_e|0_f,mufc/mudf/1_d|0_f} {
      \draw[->] (\i) -- node [inner sep=1pt] {\small $\l$} (\j);
    };
    \foreach\i/\j\l in {a/b/1_b,b/a/1_a,b/c/0_c,c/b/0_b,d/e/0_e,e/d/0_d,e/f/1_f,f/e/1_e,muad/mube/1_b|0_e,mube/muad/1_a|0_d} {
      \draw[->,decorate,decoration={single line,raise=2pt}] (\i) -- node {\small $\l$} (\j);
    };
    \foreach\i/\j/\l in {c/a/near start,b/a/very near start,b/e/,a/d/near start,c/e/} {
      \draw[->,decorate,decoration={single line,raise=2pt}] (mu2\i\j) -- node [\l] {\small $*_\i|*_\j$} (mu\i\j);
      \draw[->,decorate,decoration={single line,raise=-2pt}] (mu2\i\j) -- (mu\i\j);
    };    
  \end{fsa}
  \end{equation}

  Consider $x=p_{w,h}.q\in\widehat L(\mathscr A)$, with $q=q_0 r$ for
  an edge $q_0$ starting at $h$. Then
  $\sigma\widehat\mu^2(x)=p_{w,h_0}q_0'.\mu(r)$ for an edge $q_0'$
  starting at $h_0$ which is minimal if and only if $q_0$ is minimal;
  writing $s\in\{0,1\}$ be the label of $q_0$ without its decoration,
  and $h'$ for the initial vertex of $\mu(r)$, we get
  $\sigma\widehat\mu^2(x)=p_{w s,h'}.\mu(r)$.  On the other hand,
  $\widehat\mu\sigma(x)=\widehat\mu(p_{w,h}q_0.r)=\widehat\mu(p_{w
    s,h_1}.r)$ with $h_1$ the initial vertex of $r$, since $q_0$ is
  minimal if and only if $q_0'$ is minimal, so
  $\widehat\mu\sigma(x)=p_{w s,h'}.\mu(r)$. We thus have
  $\sigma\circ\widehat\mu^2=\widehat\mu\circ\sigma$.
\end{proof}

\subsection{Solenoids}
We embark on a quick detour of the $2$-adic solenoid. We start by the
short exact sequence
\[\begin{tikzcd}
    0\ar[r] & \Z\ar[r] & \Z[\tfrac12]\ar[r] & \Z[\tfrac12]/\Z\ar[r] & 0,
\end{tikzcd}\]
and apply Pontryagin duality (the dual of an Abelian group $A$ is
$\widehat A\coloneqq\operatorname{Hom}(A,\R/\Z)$) to obtain
\[\begin{tikzcd}
    0\ar[r] & \Z_2\ar[r] & \mathbb S_2\ar[r] & \R/\Z\ar[r] & 0.
  \end{tikzcd}
\]
Here $\mathbb S_2$, the Pontryagin dual of the discrete group
$\Z[\frac12]$, may be defined as $(\Z_2\times\R)/\Z$, with
antidiagonal action of $\Z$ on $\Z_2\times\R$, given by
$n\cdot(z,x)=(z+n,x-n)$. It is thus the suspension (a.k.a.\ mapping
torus) of $\Z_2$, namely $(\Z_2\times[0,1])/(z,1)\sim(z+1,0)$. Recall
that a group is discrete if and only if its dual is compact, and then
is torsion-free if and only if its dual is connected; and a group is
separable if and only if its dual is metrizable.

The Baumslag-Solitar group acts diagonally on $\Z_2\times\R$, by the
usual affine action: $\alpha(x)=x+1$ and $\zeta(x)=2x$, with
$x\in\Z_2$ or $x\in\R$. The beauty is that $\zeta$ is a contraction on
$\Z_2$ while an expansion on $\R$, so it induces on $\Z_2\times\R$,
and hence on $\mathbb S_2$, a hyperbolic map. On the suspension
$(\Z_2\times[0,1])/(z,1)\sim(z+1,0)$ we see
\[\alpha(z,x)=(z+2,x),\qquad\zeta(z,x)=\begin{cases}(2z,2x) & \text{ if }x\le1/2,\\(2z+1,2x-1) & \text{ if }x\ge1/2.\end{cases}
\]

Note that $\mathbb S_2$ is not the natural extension of
$(\Z_2,\zeta^{-1})$ (which we identified with $\{0,1\}^\Z$), but is a
quotient of it. Quite to the contrary of natural extensions, there is
a map in the opposite direction, $\Z_2\to\mathbb S_2$, given by
$z\mapsto(z,0)$. We prefer to consider a conjugate embedding
$z\mapsto(2z,0)$; it interlaces the action of $B(1,2)_+$ on $\Z_2$
with that of $B(1,2)$ on $\mathbb S_2$, via the obvious inclusion
$B(1,2)_+\subset B(1,2)$.

In summary, we have a space $\mathbb S_2$ which fibres over the circle
$\R/\Z$ with fibre $\Z_2$ and is foliated by real lines; parallel
transport along the circle induces on the fibre $\Z_2$ the dynamics
$\tau$; and the contracting dynamics $\zeta$ on $\Z_2$ may be combined
with the degree-$2$ covering on the circle to yield a hyperbolic
homeomorphism.

More precisely, by ``combined'' we mean that locally $\mathbb S_2$
decomposes as a product of a stable (Cantor set) direction and an
unstable (real interval) direction; this decomposition is preserved by
the dynamics $\zeta$. Consider the fixed point $(0,0)$. On a stable
direction $(\Z_2,0)$ of $\zeta$ which corresponds to the fibres of the
covering, the map $\zeta^{-1}$ is a well-defined expanding map, while
on the unstable direction $0\times(-1/2,1/2)$ which is a local section
of the covering, the map $\zeta$ expands by a factor of $2$.

\subsection{The Thue-Morse solenoid}
We are about to construct a space $\mathfrak S$ admitting much of the
properties of $\mathbb S_2$ mentioned in the previous paragraph: it
will be a topological space fibering over the circle and foliated by
real lines, equipped with an action of $B(1,2)$, containing a copy of
$(L(\mathscr A),\mu)$ as a fibre, such that the monodromy around the
circle induces the map $\mu$ on the fibre, and will be a quotient of
$\widehat L(\mathscr A)$.  Furthermore all these maps are equivariant
with respect to the actions of $B(1,2)$, respectively $B(1,2)_+$.

In fact, for the construction of $\mathfrak S$ a sizeable part of
Nekrashevych's theory of ``iterated monodromy groups'' may be
used. The point is that the automaton
$\mathscr N\coloneqq\mathscr A\cup\mathscr A_\mu\cup\mathscr
A_{\mu^{-1}}$ is ``nuclear'': for every $n\in\Z$, the recurrent states
of the automaton $\mathscr A_{\mu^n}$ are contained in $\mathscr
N$. By induction, it suffices to check this property for $n=2$, and
this is given by the automaton~\eqref{eq:mu2}.

We may thus define the space $\mathfrak S$: it is the quotient of
$\widehat L(\mathscr A)$ by ``asymptotic equivalence'', the relation
in which one declares two bi-infinite paths $w,w'$ to be equivalent if
there exists a bi-infinite path in $\mathscr N$ labeled $w|w'$.

\begin{thm}\label{thm:solenoid}
  The space $\mathfrak S=\widehat L(\mathscr A)/{\sim}$
  \begin{enumerate}
  \item is compact, metrizable, connected;
  \item fibres over the circle $\R/\Z$ with fibre $L(\mathscr A)$;
    monodromy around the circle induces the map $\mu$ on the fibre;
  \item admits an action of $B(1,2)$ in which $\zeta$ is
    ``hyperbolic'': locally $\mathfrak S$ decomposes as a product
    $V_s\times V_u$ in ``stable'' and ``unstable'' directions
    compatible with the action of $B(1,2)$; the action of $\zeta$ is
    contracting on the $V_s$ (which are contained in fibres of the
    fibration) and expanding on the $V_u$ (which are sections of the
    fibration);
  \item admits a quotient map to $\mathbb S_2$, induced by forgetting
    alphabet decorations; fibres have cardinality $2$ or $4$. There is
    a homeomorphism
    \[\mathfrak S\longrightarrow\Z/2\cdot\bigg(\frac{\widetilde\Z_2\times[0,1]}{(z,1)\sim(z+1,0)}\bigg),\]
    on the image of which the $B(1,2)$-action is given by ``affine
    action with cocycle'', and on which the map
    $\mathfrak S\twoheadrightarrow\mathbb S_2$ is given by
    $\widetilde\Z_2\to\Z_2$ and $\Z/2\to1$.
  \end{enumerate}
  Furthermore, all the maps are equivariant with respect to the
  available actions of $B(1,2)$ or $B_+(1,2)$.
\end{thm}
\begin{proof}
  First, let us check that the relation ``$w\sim w'\Leftrightarrow$
  there exists a bi-infinite path in $\mathscr N$ labeled $w|w'$'' is
  an equivalence relation: this follows because $\mathscr N$ is
  nuclear. Indeed only transitivity needs to be checked; assume there
  is in $\mathscr N$ a path labeled $w|w'$ and a path labeled
  $w'|w''$. There is then a path $p$ labeled $w|w''$ in
  $\mathscr N^2$. This path can never enter a state labeled $\mu^2$ or
  $\mu^{-2}$, since in $\mathscr N^2$ there is no edge entering such
  states; so $p$ remains in states $\mu,1,\mu^{-1}$ and thus lies in
  $\mathscr N$.

  (1) The set of paths in $\mathscr N$ defines a closed subset of
  $\widehat L(\mathscr A)\times\widehat L(\mathscr A)$, since all
  states of $\mathscr N$ are final (see
  Proposition~\ref{prop:closed}), so the quotient is compact and
  metrizable. Its connectedness will follow from (2), since the circle
  is connected and the action on the fibre is minimal.

  (2) There is a natural map $\mathfrak S\to\mathfrak J$ given by
  restriction and forgetting-of-decorations to the labels on the
  negative part: $p_{w,h}.q\mapsto w$. This defines a map to a
  quotient $\mathfrak J$ of $\{0,1\}^{-\N}$. Now whenever
  $w=1^{-\omega}0x_{-n}\cdots x_{-1}$ and
  $w'=0^{-\omega}1x_{-n}\cdots x_{-1}$ there is a path in
  $\mathscr A_\mu$ labeled $w|w'$ and ending in a state
  $h\in\{a,\dots,f\}$; and conversely every path ending in such a
  state $h$ has a label of this form. There are finally paths labeled
  $1^{-\omega}|0^{-\omega}$ ending in $\mu_{ij}$. These are precisely
  the identifications between binary sequences representing the same
  element in $\R/\Z$; so $\mathfrak J=\R/\Z$ and the fibering map is
  $p_{w,h}.q\mapsto\sum_{n<0}w_n2^n$.

  Every fibre is naturally identified with a right-infinite path in
  $\mathscr A$, namely with an element of $L(\mathscr A)$. The
  monodromy action along the circle is given by the above
  left-infinite paths ending in a state $\mu_{ij}$, since they
  correspond to the identification $1=0$ in $\R/\Z$; these paths
  continue to right-infinite paths giving the action of $\mu$ on
  $L(\mathscr A)$.

  (3) We first check that the action of $B(1,2)$ on
  $\widehat L(\mathscr A)$ descends to $\mathfrak S$. This is obvious
  for $\zeta$, since $\sim$ is shift-invariant. For $\widehat\mu$,
  consider sequences $p_{w,h}.q\sim p_{w',h'}.q'$, and write
  $\widehat\mu(p_{w,h}.q)=p_{w,k}.\mu(q)$. If $q=q'$, then there is a
  left-infinite path in $\mathscr N$ ending at $h=h'$ and labeled
  $w|w'$; then
  $\widehat\mu(p_{w',h'}.q')=p_{w',k}.\mu(q)\sim p_{w,k}.\mu(q)$;
  while if $q\neq q'$ then $q'=\mu^{\pm1}(q)$ and again there is a
  path in $\mathscr N$ ending at a state $\mu^{\pm1}$ and labeled
  $w|w'$; and then
  $\widehat\mu(p_{w',h'}.q')=p_{w',k'}.\mu(q')=p_{w',k'}.\mu^{1\pm1}(q)\sim
  p_{w,k}.\mu(q)$.

  Now the dynamics induced by $\zeta$ on left-infinite paths is
  $p_{\dots w_{-2}w_{-1},h}\mapsto p_{\dots w_{-2},h'}$, and induces
  angle doubling on $\R/\Z$, so is expanding; while on the invariant
  fibre $L(\mathscr A)$ it is $z\mapsto 0_h z$, prepending an edge
  labeled `$0$' to paths, so is contracting. The space $\mathfrak S$
  is covered by $6$ open sets, namely for all $h\in\{a,\dots,f\}$ the
  set of all paths passing through $h$ at time $0$, and decomposes on
  each such subset as a product of stable and unstable varieties:
  through $p_{w,h}.q$ they are respectively
  $V_s=\{p_{w,h}.q':q'\in L(\mathscr A)\text{ starting at }h\}$ and
  $V_u=\{p_{w',h}.q:w'\in\{0,1\}^{-\N}\}$.

  (4) The solenoid $\mathbb S_2$ may be viewed as the quotient of
  $\{0,1\}^\Z$ by the ``asymptotic equivalence'' relation given
  by~\eqref{eq:tau}. The alphabet map $i_j\mapsto i$ and the
  ``morphism of automata''
  $\mathscr A_\mu\twoheadrightarrow\mathscr B_\tau$ induce the map
  $\mathfrak S\to\mathbb S_2$.

  Now $\mathbb S_2$ and $\mathfrak S$ are both suspensions,
  respectively of $\Z_2$ and $\Z/2\cdot\widetilde\Z_2$, so the last
  claim follows.
\end{proof}

We add the remark that, for free Abelian groups, Pontryagin and
Nekrashevych duality are essentially the same (or more precisely
dual). Let us highlight the tautologies involved: let $A$ be an
Abelian group, equipped with an expanding self-map
$T\colon A\righttoleftarrow$. On one hand, one constructs an automatic
action of $A$ on $(A/T(A))^\N$ by acting on coset spaces $A/T^n(A)$
(which correspond to clopens in $(A/T(A))^\N$). Expansivity of $T$
implies the existence of a finite automaton $\mathscr N$ (the
``nucleus'') containing the recurrent subautomaton of the action of
every element of $A$.

Bi-infinite paths in $\mathscr N$ define an equivalence relation on
$(A/T(A))^\Z$, the quotient of which is a solenoid
$\mathfrak S=\widehat{A[T^{-1}]}$. A leaf in $\mathfrak S$ is a real
vector space $V$, on which $A$ acts by translation. Left-infinite
paths in $\mathscr N$ define an equivalence relation on
$(A/T(A))^{-\N}$, the quotient of which is a torus, homeomorphic to
$V/A$. We naturally have $A=\pi_1(V/A)$, and $V$ is the universal
cover of $V/A$.

On the other hand, Pontryagin duality associates with $A$ the torus
$V^*/A^\perp$, since a representation of $A$ in $\R/\Z$ extends
uniquely to $V$, and the orthogonal lattice of $A$ in $V^*$ is by
definition the set of representations that are trivial on $A$.

Nekrasheych's theory also applies also to non-Abelian groups $G$;
given a group $G$, a finite-index subgroup $H$ and a homomorphism
$\theta\colon H\to G$, one obtains (by choosing a transversal of $H$
in $G$) an action of $G$ on $(G/H)^\N$. If $\theta$ is contracting,
then this action is automatic, and also admits a finite nucleus. The
``limit space'' is $(G/H)^{-\N}/{\sim}$, a compact, metrizable space
$\mathfrak X$ equipped with a self-covering map $\sigma$ induced by
the shift. One shouldn't expect $G$ to act on $\mathfrak X$, but the
fundamental group of $\mathfrak X$ acts on iterated fibres of $\sigma$
by monodromy, and lets one recover in this manner the action of $G$ on
$(G/H)^\N$.

In our situation, we cannot directly apply this construction: our
action of $\Z$ on $L(\mathscr A)$ does not permute clopens as is the
case above; and the ``asymptotic equivalence relation'' defines
different kinds of ``glue'', depending on the state of $\mathscr N$ in
which the left-infinite path in $\mathscr N$ ends; different kinds of
glues should be applied to different portions of the space of
left-infinite paths in $\mathscr A$. Said differently, the best one
can hope for is a space $\mathfrak X$ equipped with a collection of
partial self-coverings. In our case, the space $\mathfrak X$ is
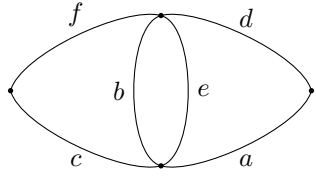
\[\begin{tikzpicture}
    \coordinate (b) at (0,0);
    \coordinate (t) at (0,2);
    \coordinate (l) at (-2,1);
    \coordinate (r) at (2,1);
    \foreach \c in {b,t,l,r} { \fill (\c) circle (1pt); }
    \draw (b) .. controls +(165:0.5) and +(-165:0.5) .. node[left] {$b$} (t)
    (t) .. controls +(-15:0.5) and +(15:0.5) .. node[right] {$e$} (b)
    (b) .. controls +(-15:0.5) and +(-105:0.5) .. node[below] {$a$} (r)
    (r) .. controls +(105:0.5) and +(15:0.5) .. node[above] {$d$} (t)
    (t) .. controls +(165:0.5) and +(75:0.5) .. node[above] {$f$} (l)
    (l) .. controls +(-75:0.5) and +(-165:0.5) .. node[below] {$c$} (b);
  \end{tikzpicture}
\]
and there is a self-map $\zeta\colon\mathfrak X\righttoleftarrow$,
which is almost a self-covering, but has branching points at all four
vertices of $\mathfrak X$. It is given on the edges by $a\mapsto d b$
etc.\ as in~\eqref{eq:extendzeta}:
\[\begin{tikzpicture}
    \clip (-3,-1.3) rectangle (3,1.3); 
    \coordinate (b) at (0,-1);
    \coordinate (t) at (0,1);
    \coordinate (l) at (-2,0);
    \coordinate (r) at (2,0);
    \foreach \c in {b,t,l,r} { \fill (\c) circle (1pt); }
    \draw (l) .. controls +(-45:0.7) and +(-180:0.7) .. node[above] {$e$} (b) .. controls +(0:0.7) and +(-135:0.7) .. (r)
    (r) .. controls +(135:0.7) and +(0:0.7) .. node[below] {$b$} (t) .. controls +(-180:0.7) and +(45:0.7) .. (l);
    \foreach \s/\t/\pos/\l/\x/\y in {r/b/above/a/-/+,l/b/above/d/+/+,r/t/below/c/-/-,l/t/below/f/+/-} {
      \draw (\s) .. controls +(\x2.5,\y3.1) and +(\x1.2,\y0.2) .. node[pos=0.1,\pos] {$\l$} (\t);
      }
  \end{tikzpicture}
\]
The reason it is not really a covering is that the connections
$\{b,c\}\to\{a,e\}$ in $\mathfrak X$ should not always be present;
indeed $\mathscr A_\mu$ identifies $p_{1^\omega,b}$ and
$p_{0^\omega,a}$ only when followed by $0_c|1_c$, identifies
$p_{1^\omega,b}$ and $p_{0^\omega,e}$ when followed by $1_a|0_d$,
identifies $p_{1^\omega,c}$ and $p_{0^\omega,a}$ when followed by
$0_b|1_b$, and identifies $p_{1^\omega,c}$ and $p_{0^\omega,e}$ when
followed by $0_f|1_f$.

\subsection{K-theory}
It is often enriching, for example so as to classify them, to compute
the K-theory (a.k.a.\ dimension groups) of dynamical systems. These
may be defined from the C*-algebra generated by the dynamical system,
but we avoid all details to jump to our case of a self-map $\sigma$ of
a totally-disconnected compact space $\Omega$; its $K_0$ group may be
defined as
\[K_0(\Omega,\sigma)=\frac{\{f\colon\Omega\to\Z\text{ continuous}\}}{f\sim f\circ\sigma},
\]
equipped with a ``positive cone'' $K_0^+$, the image of
$\{f\colon\Omega\to\N\}$, and a distinguished element $f\equiv 1$;
see~\cite{herman-putnam-skau:bratteli}*{Theorem~1.4}. We compute it,
using~\cite{herman-putnam-skau:bratteli}*{Theorem~5.4}, as
\[K_0(L(\mathscr A),\sigma)=\injlim (\Z^6,\mathscr
  A)=\Z^4\times\Z[\frac12];
\]
here the $\Z^6$ corresponds to the $6$ vertices of $\mathscr A$, and
in the colimit the `$\mathscr A$' is the adjacency matrix of
$\mathscr A$. The positive cone is
$K_0^+=\{0\}\cup(\Z^4\times\Z_+[\frac12])$.

By classical results, $K_0(\Omega,\sigma)$ is also the
$\Z$-equivariant \v Cech cohomology of the solenoid
$(\mathfrak S,\widehat\mu)$.

\section{Biminimal actions}
It is easy to see on $\mathscr A$ that $L(\mathscr A)$ is minimal: we
are to check that the operation ``take the successor'' is transitive
on finite paths. This is just the connectedness of $\mathscr A$,
namely that there does not exist any proper subautomaton.

\begin{defn}
  Let $G$ be a group acting on a topological space $\Omega$.  The
  action is \emph{biminimal} if
  \[\forall(x,y\in X)\forall(\mathcal U,\mathcal V\text{ open in
    }X)\exists (g\in G):(x\in\mathcal U,y\in\mathcal
    V)\Rightarrow(g(x)\in\mathcal V,g^{-1}(y)\in\mathcal U).\]

  Rephrasing, for every $x,y\in X$ and respective open neighbourhoods
  $\mathcal U,\mathcal V$, there exists $g\in G$ taking $x$ into
  $\mathcal V$ and $\mathcal U$ over $y$.
\end{defn}

Here and below, for $s\in\{a,\dots,f\}$ we use the notation
`$s L(\mathscr A)$' to denote those elements of $L(\mathscr A)$ that
start by the vertex $s$.

\begin{thm}\label{thm:bimin}
  The action of $\Z=\langle\mu\rangle$ on $L(\mathscr A)$ is not
  biminimal.

  More precisely, consider $x=(0_e0_d)^\omega$ and $y=(0_d0_e)^\omega$
  two minimal paths; and surround them by the open sets
  $\mathcal U=e L(\mathscr A)$ and $\mathcal V=d L(\mathscr
  A)$. Then we claim that there is no $n\in\Z$ with
  $\mu^n(x)\in\mathcal V$ and $\mu^{-n}(y)\in\mathcal U$.
\end{thm}

The proof will occupy the remainder of this section.  For
$z\in L(\mathscr A)$ and $s\in\{a,\dots,f\}$, define
\[\Lambda(z,s)=\{n\in\Z:\mu^n(z)\in s L(\mathscr A)\};\]
then $\{\Lambda(z,a),\dots,\Lambda(z,f)\}$ forms for all
$z\in L(\mathscr A)$ a partition of $\Z$, and our claim amounts to
checking that $-\Lambda(x,d)\cap\Lambda(y,e)=\emptyset$.

Let $\delta\colon\Sigma\to\{0,1\}$ be the map $i_j\mapsto i$ that
forgets letters' decorations. Write $z=z_0 z'$; then we have equations
\[\Lambda(z,s)=\bigsqcup_{(s\to t)\in\Sigma}2\Lambda(z',t)+\delta(s\to t)-\delta(z_0).
\]
Indeed if $z_0$ is minimal then $\mu^{2n}(z)=w_0\;\mu^n(z')$ with
$w_0$ minimal and $\mu^{2n+1}(z)=w_0\;\mu^n(z')$ with $w_0$
maximal; and similar equations hold if $z_0$ is maximal.

We thus have equations
\begin{align*}
  \Lambda(x,a) &=(2\Lambda(y,b)+1)\cup(2\Lambda(y,c)+1),\\
  \Lambda(x,b) &=(2\Lambda(y,c))\cup(2\Lambda(y,a)+1),\\
  \Lambda(x,c) &=(2\Lambda(y,b))\cup(2\Lambda(y,f)),\\
  \Lambda(x,d) &=(2\Lambda(y,a))\cup(2\Lambda(y,e)),\\
  \Lambda(x,e) &=(2\Lambda(y,d))\cup(2\Lambda(y,f)+1),\\
  \Lambda(x,f) &=(2\Lambda(y,d)+1)\cup(2\Lambda(y,e)+1),\\
\end{align*}
and exactly the same equations with $x,y$ interchanged. These determine the sets $\Lambda(x,s)$ and $\Lambda(y,s)$ once one specifies the ``initial values''
\[0\in\Lambda(x,e),\quad-1\in\Lambda(x,b),\quad 0\in\Lambda(y,d),\quad-1\in\Lambda(y,a).\]

Let us change notations, and read integers in binary, LSB first
(positive numbers end in $0^\infty$ and negative ones end in
$1^\infty$); then the following automaton recognizes $\Lambda(x,s)$
and $\Lambda(y,s)$ when given initial state $s$, and for appropriate
choices of final states depending on whether we wish to recognize
$\Lambda(x,s)$ or $\Lambda(y,s)$:
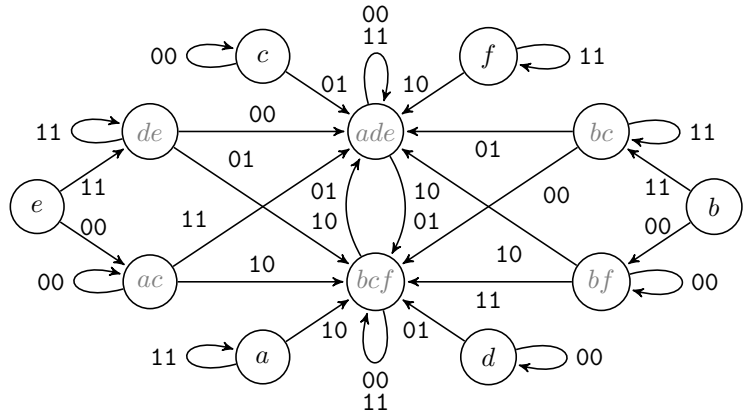
\[\begin{fsa}[every state/.style={minimum size=7mm,text width=7mm,inner sep=0mm,align=center}]
    \node[state] (110) at (-1.5,2) {$c$};
    \node[state] (011) at (1.5,2) {$f$};
    \node[state] (*01) at (-3,1) {\textcolor{gray}{$de$}};
    \node[state] (*0*) at (0,1) {\textcolor{gray}{$ade$}};
    \node[state] (*10) at (3,1) {\textcolor{gray}{$bc$}};
    \node[state] (101) at (-4.5,0) {$e$};
    \node[state] (010) at (4.5,0) {$b$};
    \node[state] (10*) at (-3,-1) {\textcolor{gray}{$ac$}};
    \node[state] (*1*) at (0,-1) {\textcolor{gray}{$bcf$}};
    \node[state] (01*) at (3,-1) {\textcolor{gray}{$bf$}};
    \node[state] (100) at (-1.5,-2) {$a$};
    \node[state] (001) at (1.5,-2) {$d$};
    
    \path (110) edge[loop left] node {$\8\8$} (110) edge node[above right=-1mm] {$\8\9$} (*0*);
    \path (011) edge[loop right] node {$\9\9$} (011) edge node[above left=-1mm] {$\9\8$} (*0*);

    \path (*01) edge[loop left] node {$\9\9$} (*01) edge node {$\8\8$} (*0*) edge node[near start] {$\8\9$} (*1*);
    \path (*0*) edge[loop above] node[above=3mm] {$\8\8$} node {$\9\9$} (*0*) edge[bend left] node[right,pos=0.35] {$\9\8$} node[right,pos=0.65] {$\8\9$} (*1*);
    \path (*10) edge[loop right] node {$\9\9$} (*10) edge node {$\8\9$} (*0*) edge node[near start] {$\8\8$} (*1*);

    \path (101) edge node[below] {$\9\9$} (*01) edge node[above] {$\8\8$} (10*);
    \path (010) edge node[below] {$\9\9$} (*10) edge node[above] {$\8\8$} (01*);
    
    \path (10*) edge[loop left] node {$\8\8$} (10*) edge node[near start] {$\9\9$} (*0*) edge node {$\9\8$} (*1*);
    \path (*1*) edge[loop below] node[below=3mm] {$\9\9$} node {$\8\8$} (*1*) edge[bend left] node[left,pos=0.35] {$\9\8$} node[left,pos=0.65] {$\8\9$} (*0*);
    \path (01*) edge[loop right] node {$\8\8$} (01*) edge node[near start] {$\9\8$} (*0*) edge node {$\9\9$} (*1*);

    \path (100) edge[loop left] node {$\9\9$} (100) edge node[below right=-1mm] {$\9\8$} (*1*);
    \path (001) edge[loop right] node {$\8\8$} (001) edge node[below left=-1mm] {$\8\9$} (*1*);
  \end{fsa}
\]
(The shaded state $de$ recognizes $\Lambda(x,d)\cup\Lambda(x,e)$,
etc.)  Since we are interested in recognizing integers, we only accept
infinite strings over $\{\8,\9\}$ that end in
$\8^\omega\cup\9^\omega$. The remaining data that need be specified
are that, to accept $\Lambda(x,s)$, one should start at state $s$ in
the above diagram and accept only $\8^\omega$ at state $ade$ and
$\9^\omega$ at state $bcf$; while to accept $\Lambda(y,s)$ one should
accept only $\8^\omega$ and $\9^\omega$ at $ade$.


Thus for example in terms of regular expressions
$\Lambda(x,d)=(\8\8)^*\8\9E\9^\omega\cup(\8\8)^*\8\9\8^*1E\8^\omega$,
where $E=\varepsilon\cup\8^*\9\8^*\9E$ is the set of words with an
even number of $\9$'s (``evil numbers''); here is the corresponding
automaton, with initial and final vertices indicated by free incoming
(respectively outgoing) arrow, and $\epsilon$ denotes an empty
transition:
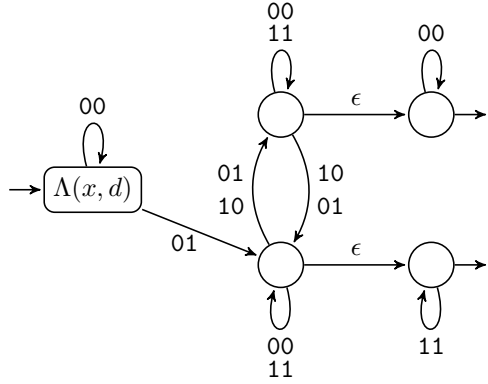
\[\begin{fsa}[every state/.style={minimum size=6mm}]
    \node[state,rectangle,rounded corners,initial] (L) at (-2.5,0) {$\Lambda(x,d)$};
    \node[state] (*0*) at (0,1) {};
    \node[state] (*1*) at (0,-1) {};
    \node[state,accepting] (1i) at (2,1) {};
    \node[state,accepting] (0i) at (2,-1) {};

    \path (L) edge[loop above] node {$\8\8$} (L) edge node[below left=-1mm] {$\8\9$} (*1*);
    \path (*0*) edge[loop above] node[above=3mm] {$\8\8$} node {$\9\9$} (*0*) edge[bend left] node[right,pos=0.35] {$\9\8$} node[right,pos=0.65] {$\8\9$} (*1*) edge node {$\epsilon$} (1i);
    \path (*1*) edge[loop below] node[below=3mm] {$\9\9$} node {$\8\8$} (*1*) edge[bend left] node[left,pos=0.35] {$\9\8$} node[left,pos=0.65] {$\8\9$} (*0*) edge node {$\epsilon$} (0i);
    \path (1i) edge[loop above] node {$\8\8$} (1i);
    \path (0i) edge[loop below] node {$\9\9$} (0i);
  \end{fsa}
\]

We also easily compute $\Lambda(y,e)$:
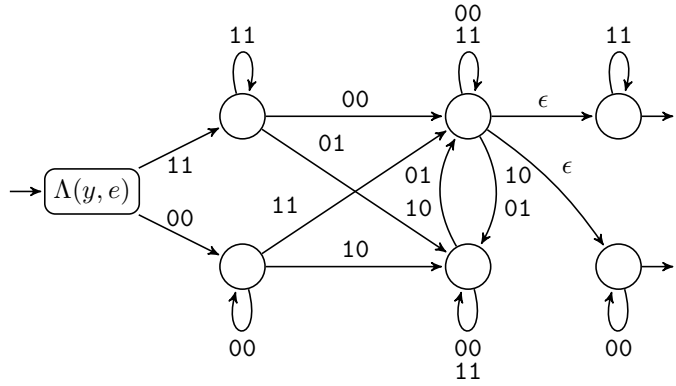
\[\begin{fsa}[every state/.style={minimum size=6mm}]
    \node[state] (*01) at (-3,1) {};
    \node[state] (*0*) at (0,1) {};
    \node[state,rectangle,rounded corners,initial] (L) at (-5,0) {$\Lambda(y,e)$};
    \node[state] (10*) at (-3,-1) {};
    \node[state] (*1*) at (0,-1) {};
    \node[state,accepting] (1i) at (2,1) {};
    \node[state,accepting] (0i) at (2,-1) {};
    
    \path (*01) edge[loop above] node {$\9\9$} (*01) edge node {$\8\8$} (*0*) edge node[near start] {$\8\9$} (*1*);
    \path (*0*) edge[loop above] node[above=3mm] {$\8\8$} node {$\9\9$} (*0*) edge[bend left] node[right,pos=0.35] {$\9\8$} node[right,pos=0.65] {$\8\9$} (*1*) edge [bend left=10] node {$\epsilon$} (0i) edge node {$\epsilon$} (1i);
    \path (L) edge node[below] {$\9\9$} (*01) edge node[above] {$\8\8$} (10*);    
    \path (10*) edge[loop below] node {$\8\8$} (10*) edge node[near start] {$\9\9$} (*0*) edge node {$\9\8$} (*1*);
    \path (*1*) edge[loop below] node[below=3mm] {$\9\9$} node {$\8\8$} (*1*) edge[bend left] node[left,pos=0.35] {$\9\8$} node[left,pos=0.65] {$\8\9$} (*0*);
    \path (1i) edge[loop above] node {$\9\9$} (1i);
    \path (0i) edge[loop below] node {$\8\8$} (0i);
  \end{fsa}
\]

Now it is easy to compute, using the automata, $-\Lambda(x,d)$; it just amounts to switching $\8^\omega$ and $\9^\omega$:
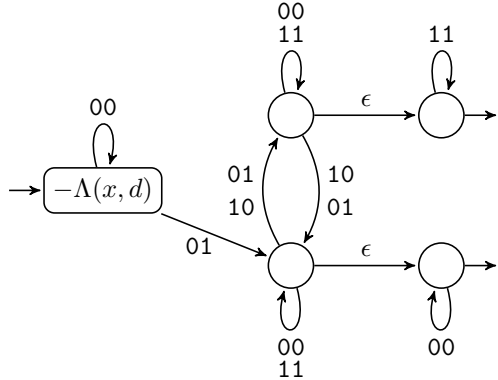
\[\begin{fsa}[every state/.style={minimum size=6mm}]
    \node[state,rectangle,rounded corners,initial] (L) at (-2.5,0) {$-\Lambda(x,d)$};
    \node[state] (*0*) at (0,1) {};
    \node[state] (*1*) at (0,-1) {};
    \node[state,accepting] (1i) at (2,1) {};
    \node[state,accepting] (0i) at (2,-1) {};

    \path (L) edge[loop above] node {$\8\8$} (L) edge node[below left=-1mm] {$\8\9$} (*1*);
    \path (*0*) edge[loop above] node[above=3mm] {$\8\8$} node {$\9\9$} (*0*) edge[bend left] node[right,pos=0.35] {$\9\8$} node[right,pos=0.65] {$\8\9$} (*1*) edge node {$\epsilon$} (1i);
    \path (*1*) edge[loop below] node[below=3mm] {$\9\9$} node {$\8\8$} (*1*) edge[bend left] node[left,pos=0.35] {$\9\8$} node[left,pos=0.65] {$\8\9$} (*0*) edge node {$\epsilon$} (0i);
    \path (1i) edge[loop above] node {$\9\9$} (1i);
    \path (0i) edge[loop below] node {$\8\8$} (0i);
  \end{fsa}
\]

From these descriptions, it is easy to see
$(-\Lambda(x,d))\cap\Lambda(y,e)=\emptyset$, since strings in
$-\Lambda(x,d)$ start by $\8^{2n+1}\9$ while strings in $\Lambda(y,e)$
start by $\8^{2n}\9$. Thus there are no $n\in\Z$ with
$\mu^n(x)\in d L(\mathscr A)$ and $\mu^{-n}(y)\in e L(\mathscr A)$,
and Theorem~\ref{thm:bimin} is proven.

We may translate these calculations back to the original presentation
of $L(\mathscr A)$ as a subshift $\Omega\subset\{0,1\}^\Z$: the
sequences $x,y$ are
\[x=\overline u.u=\zeta(y),\qquad y=u.u=\zeta(x),
\]
surrounded by respective neighbourhoods $\mathcal U=\{*1.01*\}$ and
$\mathcal V=\{*0.01*\}$.

\section{Acknowledgments}
I am grateful to Dominic Franc\oe ur and Volodymyr Nekrashevych for
discussions on these topics.

\end{document}